\documentclass[leqno,10pt]{amsart}

\usepackage[english]{babel}
\usepackage[T1]{fontenc}

\usepackage{amsmath,amssymb,amsthm}
\usepackage{mathtools}  
\usepackage[margin=1in]{geometry}
\usepackage{float}
\usepackage{algorithm}
\usepackage{algpseudocode}
\usepackage[dvipsnames]{xcolor}

\usepackage{graphics,tikz,caption}
\usepackage{subfig} 
\usetikzlibrary{calc,positioning}

\usepackage{mathrsfs}
\usepackage{stmaryrd}

\usepackage[colorlinks	=	true, linkcolor	=	red, urlcolor	=	red, citecolor	=	red]{hyperref}
\usepackage{bookmark}									
\usepackage{cleveref}

\usepackage{multicol}
\usepackage{todonotes}
\usepackage{enumitem}
\usepackage{verbatim}

\theoremstyle{definition}
\newtheorem{thm}{Theorem}[section]
\newtheorem{exm}[thm]{Example}
\newtheorem{defi}[thm]{Definition}
\newtheorem{lemm}[thm]{Lemma}

\newtheorem{prop}[thm]{Proposition}
\newtheorem{ques}[thm]{Question}

\newtheoremstyle{case}
{3pt}
{3pt}
{}
{}
{\itshape}
{:}
{.5em}
{}

\theoremstyle{case} 
\newtheorem{case1}{Case}
\newtheorem{case2}{Case}

\theoremstyle{remark}

\DeclareMathOperator{\im}{Im}

\allowdisplaybreaks

\newcommand{\apmd}[2][]{															
	\ifthenelse{\equal{#1}{}}%
					{ \operatorname{N}_{#2}	}%
					{ \operatorname{N}_{#1}(#2) 	}}
					
\begin{document}

\title{Metric geodesic covers of graphs}
\author{Jerry Chen}
\address{Department of Mathematical Sciences, Stevens Institute of Technology}
\email{jchen9@stevens.edu}

\author{Kyle Hess}
\address{Department of Mathematics, UCLA}
\email{kghess@g.ucla.edu}

\author{Matthew Romney}
\address{Department of Mathematics, University of Hawaii at Manoa}
\email{mromney@hawaii.edu}
	
\date{}
\thanks{}

\begin{abstract}
    We study the problem of finding, for a given one-dimensional topological space $X$, a cover of $X$ of smallest size by geodesics with respect to some metric. The infimal size of such a set is called the \textit{metric geodesic cover number} of $X$. We prove reductions enabling us to find, with computer assistance, optimal geodesic covers of a graph and use these to determine the cover number of several standard graphs, including $K_4$, $K_5$ and $K_{3,3}$. We also give a catalogue of topological spaces with cover number $3$, and use it to deduce that any such space must be planar. 
\end{abstract}

\thanks{M. Romney and J. Chen were partially supported by the National Science Foundation under grant DMS-2413156.}
\maketitle

\section{Introduction}

A \textit{geodesic} is a curve of minimal length and one of the basic objects in metric geometry. In this paper, we investigate the following natural problem: given a one-dimensional topological space, such as a topological graph, to find a cover of this space by geodesics, with respect to some metric, of smallest size. We introduce the following definition.

\begin{defi} \label{defi:mgcn}
    Let $X$ be a topological space. The \textit{metric geodesic cover number} of $X$ is the infimal $n \in \mathbb{N}$ such that there is a metric on $X$ inducing the same topology in which $X$ is the union of $n$ geodesics in $X$.
\end{defi}

We also state the following variation. 

\begin{defi} \label{defi:emgcn}
    Let $X$ be a topological space. The \textit{extended metric geodesic cover number} of $X$ is the infimal $n \in \mathbb{N}$ such that there exists a topological space $Y$ containing $X$ and a metric on $Y$ inducing the same topology in which $X$ is the union of $n$ geodesics in $Y$.
\end{defi}

The second definition has the advantage that it behaves monotonically with set inclusion: if $X_1 \subset X_2$, then the extended metric geodesic cover number of $X_1$ is no greater than that of $X_2$. However, the extended metric geodesic cover number is more difficult to compute and we only have limited progress in this direction. 

To our knowledge, these definitions have not been studied before, although a variety of similar notions have been widely studied in graph theory. Our definition most closely resembles the \textit{edge geodesic cover number} of a graph. This is defined similarly but with three key differences:

\begin{itemize}
    \item[1.] For the edge geodesic cover number, geodesics may begin and end only at graph vertices. Our definition makes no such requirement.
    \item[2.] For the edge geodesic cover number, each edge implicitly has a length of $1$. In our definition, each edge can be given an arbitrary positive length.
    \item[3.] Our definition does not require the space to be a topological graph. For example, consider a metric space comprising two geodesics that intersect on a Cantor set.
\end{itemize}

Due to the these differences, the metric geodesic cover number has essential differences in behavior from the edge geodesic cover number. We use the word ``metric'' in our terminology mainly as a way to distinguish our definition from alternative definitions in the literature. For conciseness, we will usually just refer to the ``cover number'' and ``extended cover number'' of a space throughout this paper.  

Other related notions include the \textit{vertex geodesic cover number} of a graph, often called simply the \textit{geodesic cover number} or the \textit{isometric path number} in the literature. This is defined similarly to the edge geodesic cover number, but only requiring the vertices to be covered. See Manuel--Klavzar--Prabha--Arokiaraj \cite{manuel_geodesic_2024} and Chakraborty--Foucaud--Hakanen \cite{CFH:25} for some contributions to these topics and Manuel \cite{manuel_revisiting_2018} for a survey. Note that while both versions of the concept are found in the literature, the vertex version seems the more popular of the two. We refer the reader to Fisher--Fitzpatrick \cite{FF:01}, Pan--Chang \cite{pan_isometric_2003} and Manuel \cite{manuel_isometric_2018} specifically for the vertex geodesic cover number.
A further related concept is the \textit{geodetic number} of a graph, introduced by Harary--Loukakis--Tsouros in \cite{harary_geodetic_1993}. The geodetic number is the minimum size of a set \(S\) of vertices so that any other vertex is on a geodesic between two vertices from \(S\). See Bre{\v s}ar--Kov{\v s}e--Tepeh \cite{BBM:11} for a survey of this topic. Finally, our work has some similarity with the \textit{order dimension} of a graph studied by Schnyder \cite{Sch:89}.


We also mention the notion of \textit{metrizable graph}, introduced by Cizma--Linial in \cite{cizma_geodesic_2022} and further studied in \cite{chudnovsky_structure_2023}. A metrizable graph is defined to be one for which any so-called \textit{consistent family of paths} admits a length metric in which each path from the collection is a geodesic. Their work is geared at the same underlying question we consider here: determining for a given family of paths whether a space admits a metric in which they are geodesics. However, there seems to be no direct correspondence between their results and those in this paper. In particular, one conclusion from \cite{cizma_geodesic_2022} is that metrizable graphs are fairly rare. On the other hand, every finite graph necessarily admits a geodesic cover of smallest size.

This paper has three main objectives. First is to develop the basic theory of metric geodesic covers sufficiently far to enable one to compute the cover number of a given graph, at least with sufficient computation time. Our main result in this direction is the following theorem, which guarantees that any graph has an optimal geodesic cover of a specific type. See Section 3 for the terminology.

\begin{thm} \label{thm:2subdivision}
Let $X$ be a finite topological graph. Then there is a retracted geodesic cover of smallest size whose endpoints are graph vertices of a $2$-subdivision of $X$. 
\end{thm}

The second objective is to use Theorem \ref{thm:2subdivision} to actually find, with computer assistance, optimal covers for a variety of example graphs. While our computational capabilities are limited to graphs with few vertices, they nevertheless serve the role of acquiring a baseline of hard data on the topic that can guide future investigation. 

As a third objective, we provide in Proposition \ref{prop:3geodesic_catalogue} a complete catalogue of spaces with geodesic cover number three. This seems to be the largest number for which a relatively simple structure exists. While our catalogue involves several individual cases, it leads to the following consequence.

\begin{thm} \label{thm:planar}
If $X$ has cover number $3$ or less, then $X$ is planar.
\end{thm} 
In particular, any non-planar graph has extended cover number at least $4$. Propsition \ref{prop:3geodesic_catalogue} will also enable us to determine the extended cover number of $K_4$ in Proposition \ref{prop:K4}.

This project is partly inspired by the recent paper \cite{LRT:24} on so-called metric polygons. A \textit{metric polygon} is a metric space that can be written as the union of finitely many sets, each isometric to a closed interval, joined cyclically. It was shown that any metric triangle can be embedded in the plane under a bi-Lipschitz map with uniformly bounded distortion. It was asked as Question 1.5 in \cite{LRT:24} whether any metric quadrilateral is embeddable in the plane under a bi-Lipschitz mapping. We show in this paper that, while both basic non-planar graphs $K_5$ (the complete graph on five vertices) and $K_{3,3}$ (the ``three utilities graph'') have cover number $4$, these graphs nevertheless cannot be represented as a metric quadrilateral. See Section \ref{sec:K5} below. This provides evidence in favor of a positive answer to Question 1.5 of \cite{LRT:24}. 


Finally, we mention two questions for future work. For large graphs, it may be intractable to determine computational the exact cover number of a graph, and so one may be satisfied with useful upper or lower bounds. Proving upper bounds on the cover number or extended cover number is conceptually simple, only requiring the construction of an efficient geodesic cover. On the other hand, techniques to find lower bounds seem more difficult. 

\begin{ques}
    Are there broadly applicable methods to produce lower bounds on the cover number and extended cover number of a space?
\end{ques}

Next, we do not yet have any general method for computing the extended cover number of a graph. An affirmative answer to the following question, which asks for an analogue of Theorem \ref{thm:2subdivision}, would provide such a technique.

\begin{ques}
    Is there a general procedure to construct, for a given finite graph $X$, a larger finite graph $Y$ such that an optimal extended geodesic cover of $X$ necessarily appears as a subgraph of $Y$? 
\end{ques} 

Ideally, $Y$ should be as simple as possible. For example, can we assume that $Y$ is the complete graph on the vertices of a $2$-subdivision of $X$? Note that an extended geodesic cover of smallest size of $X$ can be much more complicated than the space $X$ itself. For example, the $3$-caterpillar graph described in Example \ref{exm:caterpillar} below can be covered by $2$ geodesics forming the sawtooth configuration depicted in Figure \ref{fig:caterpillar2} with an arbitrarily large number of teeth. We note that all the individual graphs considered in this paper have small cover number, allowing them to be handled through computation or \textit{ad hoc} arguments, for example, based on planarity.

The paper is structured as follows. Section \ref{sec:preliminaries} contains preliminaries on metric spaces and topological graphs needed for the paper. Section \ref{sec:examples} provides a few simple propositions and examples to illustrate the main features of the cover number and extended cover number. In Section \ref{sec:theory}, we develop techniques for determining the cover number of a graph, including the proof of Theorem \ref{thm:2subdivision}. Section \ref{sec:three_geodesics} concerns spaces with cover number $3$, culminating in the proof of Theorem \ref{thm:planar}. Section \ref{sec:computation} contains the results of our computer-assisted analysis of several graphs of interest.  Finally, there are two appendices; the first contains details of our computer implementation of an algorithm to find optimal geodesic covers, while the second includes a case analysis used in the proof of Proposition \ref{prop:3geodesic_catalogue}. 

\subsection*{Acknowledgments}

The authors are grateful to Anton Lukyanenko and the Mason Experimental Geometry Lab at George Mason University, where the authors first met and the main research for this paper was initiated.

\section{Preliminaries} \label{sec:preliminaries}

In this section, we review some of the basic concepts and fix our terminology. For more general background on metric geometry, we recommend the monograph of Bridson--Haefliger \cite{BH:99}.

We assume throughout this paper that all topological spaces are Hausdorff: for any two distinct points $x$ and $y$ in the space, there exist disjoint open sets containing $x$ and $y$, respectively. Given a topological space $X$, a \textit{metric} on $X$ is a continuous function $d \colon X \times X \to [0,\infty)$ satisfying the following properties for all $x, y, z \in X$: 
\begin{itemize}
    \item[(i)] Positive definiteness: $d(x,y) = 0$ if and only if $x = y$, 
    \item[(ii)] Symmetry: $d(x,y) = d(y,x)$, 
    \item[(iii)] Triangle inequality: $d(x,z) \le d(x,y) + d(y,z)$
\end{itemize}
While a space admits many possible metrics, we will be mainly concerned with the class of length metrics. We review the relevant terminology. A \textit{path} in $X$ is a continuous function $\gamma \colon [a,b] \to X$ for some $a<b$. The length of a path $\gamma$ is
\begin{equation} \label{equ:length}
    \ell(\gamma) = \sup \sum_{i=1}^n d(\gamma(t_{i-1}), \gamma(t_i)),
\end{equation} 
the supremum taken over all finite partitions $a = t_0 < t_1 < \ldots < t_n = b$. A metric $d$ is a \textit{length metric} if
\begin{equation} \label{equ:length_metric}
d(x,y) = \inf \ell(\Gamma),
\end{equation} 
the infimum taken over all paths $\gamma$ connecting $x$ and $y$. While Definitions \ref{defi:mgcn} and \ref{defi:emgcn} do not specify that the metrics should be a length metric, without loss of generality we may assume that this is the case. To see this, notice that if $d$ is a metric on $X$ in which $X$ is the union of finitely many geodesics, and $X$ is connected, then any pair of points in $X$ are connected by a path of finite length. The metric $d$ therefore induces a length metric $\bar{d}$ by defining $\bar{d}(x,y) = \inf \ell(\gamma)$, where the infimum is taken over all paths $\gamma$ connecting $x$ to $y$ as in \eqref{equ:length_metric}. It is easy to see that any geodesic in the original metric is still a geodesic in the induced length metric $\bar{d}$. If $X$ is not connected, we arrive at the same conclusion by allowing for $\bar{d}$ to take the value $\infty$ at some values. 

The spaces we consider in this paper are mainly topological graphs. A \textit{(topological) graph} is a pair $G = (V, E)$, where $V$ is the set of vertices (i.e., points) and $E$ is the set of edges (i.e., closed intervals whose endpoints are vertices). All graphs we consider are undirected. We identify $G$ with the topological space formed by the set of edges. More precisely, $G$ is the quotient space given by the collection of edges after identifying endpoints representing the same vertex. We do not assume that graphs are simple; that is, it is allowable for a pair of vertices to be joined by multiple edges. We also allow graphs to have loops (edges whose endpoints are the same vertex). We do not assume that a graph is connected.

The \textit{degree} (or \textit{valency}) of a vertex is the number of times that vertex appears as the endpoint of an edge. The \textit{maximum degree} of a graph is the maximum value of the degree over all vertices, provided this value exists.  

A length metric on a graph can be specified simply by assigning a length or weight $\ell_I$ to each edge $I$. To be more precise, for each edge $I$ we fix a homeomorphism $\varphi_I \colon I \to [0, \ell_I]$. For two points $x,y \in I$, we define $d_I(x,y) = |\varphi(x) - \varphi(y)|$. The length of a path is defined as follows: for a given partition, we can add points to ensure that any two consecutive points belong to the same edge. We can then define $\ell(\gamma)$ as in \eqref{equ:length} but using $d_I$, where $I$ is the common edge, in place of $d$. The metric $d$ on $X$ is then defined using \eqref{equ:length_metric}.

\section{First properties and examples} \label{sec:examples}

In this section, we present a few elementary facts and examples that illustrate the behavior of the metric geodesic cover number. First, we note a lower bound in terms of the degree of a vertex. 

\begin{prop} \label{prop:star}
    Let \(\Delta\) be the maximum degree of a finite graph \(X\). Then the extended geodesic cover number of \(X\) is at least \(\lceil \Delta/2\rceil\). In particular, the geodesic cover number of the \(n\)-star \(K_{1,n}\) is at least \(\lceil n/2\rceil\).
\end{prop}

\begin{proof}
    Consider an arbitrary geodesic cover of \(X\) within a graph \(Y\) containing $X$. Let \(v\) be a vertex in \(X\) of degree \(\Delta\), let $\alpha$ be a geodesic in $Y$, and a homeomorphism \(f:[0,L]\to Y\) with \(\im(f)=\alpha\) and \(f(r)=v\) for some \(r\in[0,L]\). As \([0,1]\setminus \{r\}\) locally has at most two connected components around \(r\), so must \(\alpha\setminus\{v\}\) around \(v\). But the graph \(X\setminus\{v\}\) locally has \(\Delta\) connected components around \(v\) to cover.
\end{proof}

\begin{figure}[t!]
\centering
\captionsetup[subfloat]{width=4 cm, justification=centering}
\subfloat[Even number of leaves]{
    \begin{tikzpicture}[scale=0.30, every node/.style={transform shape, circle, draw=black!60, minimum size=1cm, font=\bfseries}]
    \useasboundingbox (-6,-6.5) rectangle (6,6);
        \node[fill=gray!40] (A) {};
        \node[fill=green!50] (B) [above left=3cm and 4cm of A] {};
        \node[fill=red!50] (C) [above right=3cm and 4cm of A] {};
        \node[fill=blue!50] (D) [below left=3cm and 4cm of A] {};
        \node[fill=orange!50] (E) [below right=3cm and 4cm of A] {};
        \node[fill=green!50] (F) [left=5cm of A] {};
        \node[fill=orange!50] (G) [right=5cm of A] {};
        \node[fill=red!50] (H) [above=5cm of A] {};
        \node[fill=blue!50] (I) [below=5cm of A] {};
        \draw[red!70, thick] (H) -- (A) -- (C);
        \draw[orange!70, thick] (G) -- (A) -- (E);
        \draw[blue!70, thick] (I) -- (A) -- (D);
        \draw[green!70, thick] (B) -- (A) -- (F);
    \end{tikzpicture}
}
\hspace{5 cm}
\subfloat[Odd number of leaves]{
    \begin{tikzpicture}[scale=0.30, every node/.style={transform shape, circle, draw=black!60, minimum size=1cm, font=\bfseries}]
        \useasboundingbox (-6,-6.5) rectangle (6,6);
        \node[fill=gray!40] (A) {};
        \node[fill=green!50] (B) [above left=3cm and 4cm of A] {};
        \node[fill=red!50] (C) [above right=3cm and 4cm of A] {};
        \node[fill=blue!50] (D) [below left=0.5cm and 5cm of A] {};
        \node[fill=orange!50] (E) [below right=0.5cm and 5cm of A] {};
        \node[fill=blue!50] (F) [below left=4.5cm and 2cm of A] {};
        \node[fill=orange!50] (G) [below right=4.5cm and 2cm of A] {};  
        \node[fill=red!50] (H) [above=5cm of A] {};
        \draw[red!70, thick] (H) -- (A) -- (C);  
        \draw[orange!70, thick] (G) -- (A) -- (E);  
        \draw[blue!70, thick] (F) -- (A) -- (D);  
        \draw[green!70, thick] (B) -- (A); 
    \end{tikzpicture}
}
\caption{Geodesic covers of the $n$-star.}
\label{fig:nstar-both}
\end{figure}
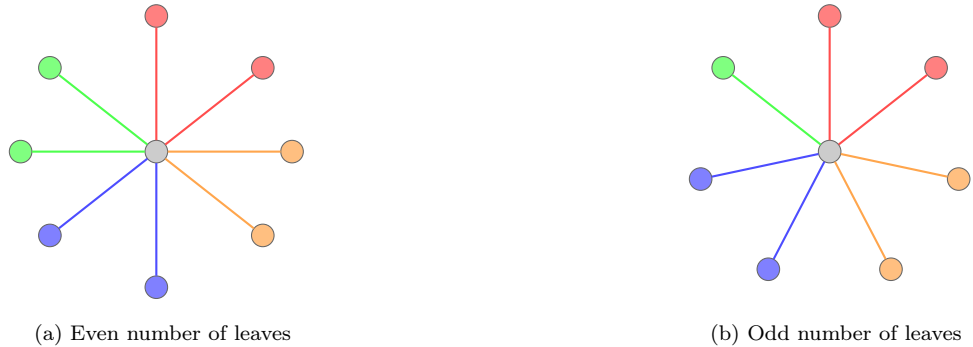

In the following, say a loop is \textit{isolated} if the vertex it is connected to has no other edges. For the following proposition, we recall that we do not assume a graph is connected. 

\begin{prop}
Let $X$ be a graph containing $m$ edges and \(k\) isolated self-loops. Then its geodesic covering number is at most $m+k$.
\end{prop} 
\begin{proof}
Assign each edge a length of $1$. On all self-loops, add a geodesic that starts at its midpoint and ends at the vertex it is connected to, and cover all other edges by a single geodesic. For each vertex without an isolated self-loop, fix an injective function from its self loops to its edges such that no self loop is sent to itself. Extend any edge in the image of this function around the half of the self-loop that was not covered. For isolated self-loops, simply add a second edge covering the other half.
\end{proof}

The preceeding proposition is sharp; for $n$ odd, take a graph formed by $n$ loops on a single vertex. The next proposition considers the case when $X$ many vertices of degree $1$.

\begin{prop} \label{prop:tree}
    Let \(X\) be a graph with \(n\) vertices of degree 1. Then the geodesic cover number of \(X\) is at least $\lceil n/2\rceil$.
\end{prop}

\begin{proof}
    If \(v\) is a vertex of degree 1, any neighborhood of \(v\) is locally homeomorphic to an endpoint of an interval. But any geodesic only has two such points (at its endpoints), and thus can only cover two such vertices.
\end{proof}

In particular, if $X$ is a tree with $n$ leaves, then the cover number of $X$ is at least $\lceil n/2\rceil$.

The following example shows that the cover number and extended cover number of a graph may be different, and in fact the gap between them may be arbitrarily large. 

\begin{exm} \label{exm:caterpillar}
    For each $n \in \mathbb{N}$, let \(X\) be the \textit{$n$-th caterpillar graph} formed by attaching an extra edge to each vertex of a simple path graph of length \(n\), as depicted in Figure \ref{fig:caterpillar1}. Then the cover number of $X$ is \(\lceil (n+1)/2\rceil\), while the extended cover number of \(X\) is 2.

    First, observe that $X$ has $n+1$ leaves, and so by Proposition \ref{prop:tree} $X$ has cover number at least \(\lceil (n+1)/2\rceil\). To see that the cover number is \(\lceil (n+1)/2\rceil\), we assign each edge length $1$. The first and last vertices are connected by a geodesic that spans the entire underlying simple path. We join the remaining pairs of leaves by a geodesic, possibly leaving a lone leaf at the end. We join this leaf to any point on the underlying simple path. This collection contains exactly \(\lceil (n+1)/2\rceil\) geodesics. We further note that Proposition \ref{prop:star} shows that the extended cover number is at least 2.

    For the upper bound, label the vertices in the bottom row in order by \(v_1,\ldots,v_n\) and those in the top row by \(u_1,\ldots,u_n\). By enlarging the space if needed, we may assume that $n$ is even. Let \(Y\) be the space \(X\) together with edges added between \(u_{2i-1}\) and \(u_{2i}\) for each \(1\leq i\leq n/2\) as shown in Figure \ref{fig:caterpillar2}. Now assign weight $3$ to each edge along the main path in \(Y\) and weight $1$ to each other edge. Let \(\alpha\) be the path traversing the bottom row. Let \(\beta\) be the concatenation of the paths \(\beta_i\) that visit vertex \(v_{2i-1}\), then \(u_{2i-1}\), then \(u_{2i}\), then \(v_{2i}\), then \(v_{2i+1}\) for each \(1\leq i\leq k\), leaving out the last edge from \(v_{2i}\) to \(v_{2i+1}\) when \(i=k\). It is easy to see that $\beta$ is a geodesic and that $\{\alpha,\beta\}$ form a geodesic cover of $X$ of size $2$.
    
\end{exm}

\begin{figure}[t!]
\centering
\captionsetup[subfloat]{width=12cm, justification=centering}
\subfloat[Caterpillar graph]{
    \begin{tikzpicture}[scale=0.5, every node/.style={transform shape, circle, draw=black!60, fill=black!30, minimum size=1cm, font=\bfseries}]
        \node (A) at (0,0) {};
        \node (B) [below = 3.25cm of A] {};
        \node (C) [right = 3.25cm of B] {};
        \node (D) [above = 3.25cm of C] {};
        \node (E) [right = 3.25cm of C] {};
        \node (F) [above = 3.25cm of E] {};
        \node (G) [right = 3.25cm of E] {};
        \node (H) [above = 3.25cm of G] {};
        \node (I) [right = 3.25cm of G] {};
        \node (J) [above = 3.25cm of I] {};
        \node (K) [right = 3.25cm of I] {};
        \node (L) [above = 3.25cm of K] {};
        \node (M) [right = 3.25cm of K] {};
        \node (N) [above = 3.25cm of M] {};
        \node (O) [right = 3.25cm of M] {};
        \node (P) [above = 3.25cm of O] {};
        \draw[very thick] (A) -- (B); 
        \draw[very thick] (B) -- (C);
        \draw[very thick] (C) -- (D);
        \draw[very thick] (C) -- (E);
        \draw[very thick] (E) -- (F); 
        \draw[very thick] (E) -- (G); 
        \draw[very thick] (G) -- (H); 
        \draw[very thick] (G) -- (I); 
        \draw[very thick] (I) -- (J); 
        \draw[very thick] (I) -- (K); 
        \draw[very thick] (K) -- (L); 
        \draw[very thick] (K) -- (M); 
        \draw[very thick] (M) -- (N); 
        \draw[very thick] (M) -- (O); 
        \draw[very thick] (O) -- (P); 
    \end{tikzpicture} \label{fig:caterpillar1}
}
\vspace{0cm} 
\subfloat[Extended caterpillar graph]{
    \begin{tikzpicture}[scale=0.5, every node/.style={transform shape, circle, draw=black!60, fill=black!30, minimum size=1cm, font=\bfseries}]
        \node (A) at (0,0) {};
        \node (B) [below = 3.25cm of A] {};
        \node (C) [right = 3.25cm of B] {};
        \node (D) [above = 3.25cm of C] {};
        \node (E) [right = 3.25cm of C] {};
        \node (F) [above = 3.25cm of E] {};
        \node (G) [right = 3.25cm of E] {};
        \node (H) [above = 3.25cm of G] {};
        \node (I) [right = 3.25cm of G] {};
        \node (J) [above = 3.25cm of I] {};
        \node (K) [right = 3.25cm of I] {};
        \node (L) [above = 3.25cm of K] {};
        \node (M) [right = 3.25cm of K] {};
        \node (N) [above = 3.25cm of M] {};
        \node (O) [right = 3.25cm of M] {};
        \node (P) [above = 3.25cm of O] {};
        \draw[red!90!black, very thick, line width = 3pt] (B) -- node[pos=0.5, sloped, above, font=\bfseries\Huge, draw=none, fill=none, inner sep=0pt] {1} (A);
        \draw[blue!90!black, very thick, line width = 3pt] (B) -- node[pos=0.5, sloped, below, font=\bfseries\Huge, draw=none, fill=none, inner sep=0pt] {3} (C);
        \draw[red!90!black, very thick, line width = 3pt] (D) -- node[pos=0.5, sloped, above, font=\bfseries\Huge, draw=none, fill=none, inner sep=0pt] {1} (C);
        \draw[red!90!black, very thick, line width = 6pt] (C) -- node[pos=0.5, sloped, above, font=\bfseries\Huge, draw=none, fill=none, inner sep=0pt] {3} (E); 
        \draw[blue!90!black, very thick, line width = 2.5pt] (C) -- node[pos=0.5, sloped, below, font=\bfseries\Huge, draw=none, fill=none, inner sep=0pt] {3} (E); 
        \draw[red!90!black, very thick, line width = 3pt] (E) -- node[pos=0.5, sloped, above, font=\bfseries\Huge, draw=none, fill=none, inner sep=0pt] {1} (F); 
        \draw[blue!90!black, very thick, line width = 3pt] (E) -- node[pos=0.5, sloped, below, font=\bfseries\Huge, draw=none, fill=none, inner sep=0pt] {3} (G); 
        \draw[red!90!black, very thick, line width = 3pt] (H) -- node[pos=0.5, sloped, above, font=\bfseries\Huge, draw=none, fill=none, inner sep=0pt] {1} (G); 
        \draw[red!90!black, very thick, line width = 6pt] (G) -- node[pos=0.5, sloped, above, font=\bfseries\Huge, draw=none, fill=none, inner sep=0pt] {3} (I); 
        \draw[blue!90!black, very thick, line width = 3pt] (G) -- node[pos=0.5, sloped, below, font=\bfseries\Huge, draw=none, fill=none, inner sep=0pt] {3} (I); 
        \draw[red!90!black, very thick, line width = 3pt] (I) -- node[pos=0.5, sloped, above, font=\bfseries\Huge, draw=none, fill=none, inner sep=0pt] {1} (J); 
        \draw[blue!90!black, very thick, line width = 3pt] (I) -- node[pos=0.5, sloped, below, font=\bfseries\Huge, draw=none, fill=none, inner sep=0pt] {3} (K); 
        \draw[red!90!black, very thick, line width = 3pt] (L) -- node[pos=0.5, sloped, above, font=\bfseries\Huge, draw=none, fill=none, inner sep=0pt] {1} (K); 
        \draw[red!90!black, very thick, line width = 6pt] (K) -- node[pos=0.5, sloped, above, font=\bfseries\Huge, draw=none, fill=none, inner sep=0pt] {3} (M); 
        \draw[blue!90!black, very thick, line width = 3pt] (K) -- node[pos=0.5, sloped, below, font=\bfseries\Huge, draw=none, fill=none, inner sep=0pt] {3} (M);
        \draw[red!90!black, very thick, line width = 3pt] (M) -- node[pos=0.5, sloped, above, font=\bfseries\Huge, draw=none, fill=none, inner sep=0pt] {1} (N); 
        \draw[blue!90!black, very thick, line width = 3pt] (M) -- node[pos=0.5, sloped, below, font=\bfseries\Huge, draw=none, fill=none, inner sep=0pt] {3} (O); 
        \draw[red!90!black, very thick, line width = 3pt] (P) -- node[pos=0.5, sloped, above, font=\bfseries\Huge, draw=none, fill=none, inner sep=0pt] {1} (O); 
        \draw[red!90!black, very thick, line width = 3pt] (A) -- node[pos=0.5, sloped, above, font=\bfseries\Huge, draw=none, fill=none, inner sep=0pt] {1} (D); 
        \draw[red!90!black, very thick, line width = 3pt] (F) -- node[pos=0.5, sloped, above, font=\bfseries\Huge, draw=none, fill=none, inner sep=0pt] {1} (H);
        \draw[red!90!black, very thick, line width = 3pt] (J) -- node[pos=0.5, sloped, above, font=\bfseries\Huge, draw=none, fill=none, inner sep=0pt] {1} (L);  
        \draw[red!90!black, very thick, line width = 3pt] (N) -- node[pos=0.5, sloped, above, font=\bfseries\Huge, draw=none, fill=none, inner sep=0pt] {1} (P);   
    \end{tikzpicture} \label{fig:caterpillar2}
}
\caption{The caterpillar graph and an extended geodesic cover of smallest size.}
\label{fig:caterpillar_stacked}
\end{figure}
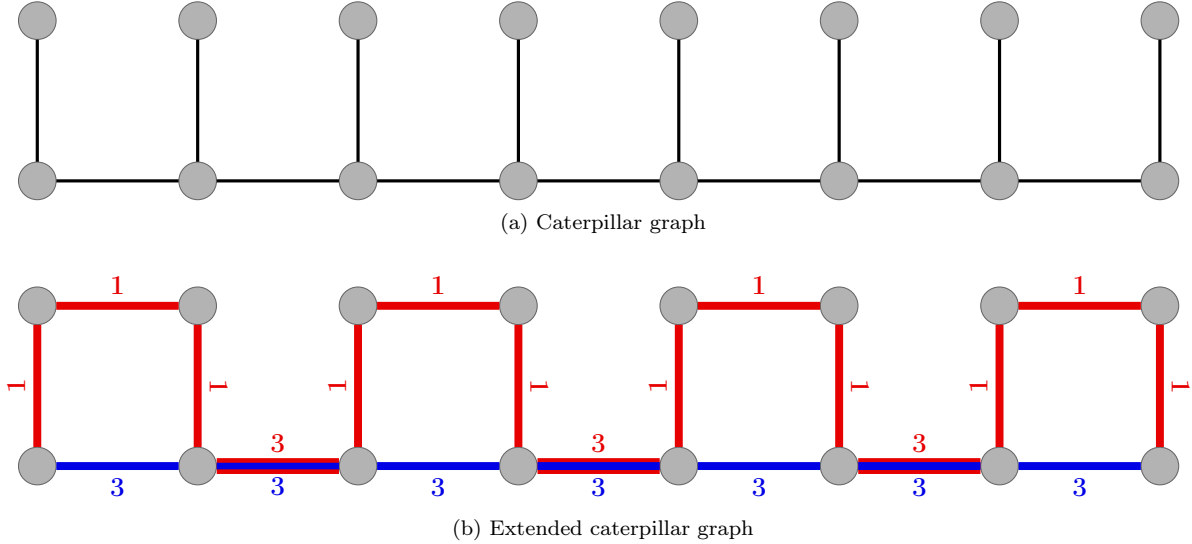

Next, what about the choice of length metric? We define the \textit{unweighted (geodesic) cover number} of a graph as in Definition \ref{defi:mgcn} but requiring each edge to have length $1$. The next example shows that the unweighted cover number may be larger than the cover number, and that the gap may be arbitrarily large.

\begin{exm}
For each $n \in \mathbb{N}$, let $X_n$ be the $n$-th \textit{sawtooth graph} formed by chaining $n$ triangles as in Figure \ref{fig:sawtooth}. The cover number of $X_n$ is two for all $n$, as can be seen by assigning a weight of $1$ to each lower edge and a weight of $1/2$ to each upper edge. On the other hand, the unweighted cover number of $X_n$ is $n$ for \(n\geq 2\).

To see the lower bound, note that no geodesic may cover both upper edges of a triangle. Thus each triangle must contain the endpoints of two geodesics. This forces \(2n\) endpoints, and thus \(n\) geodesics. To see the lower bound, label the upper edges \(u_1,\ldots,u_{2n}\) from left to right and the lower edges \(l_1,\ldots,l_n\) from left to right. For the first geodesic, start midway through \(u_1\), then proceed along all lower edges, ending midway through \(u_{2n}\). For the second geodesic, start again midway through \(u_1\), then proceed to \(u_2\), \(u_3\), and end midway through \(u_4\). For the \(i\)-th geodesic where \(i\geq 3\) (in the cases that \(n\geq 3\)), start midway through \(u_{2(i-1)}\), then proceed through \(u_{2i-1}\), then end midway through \(2i\).

\end{exm} 

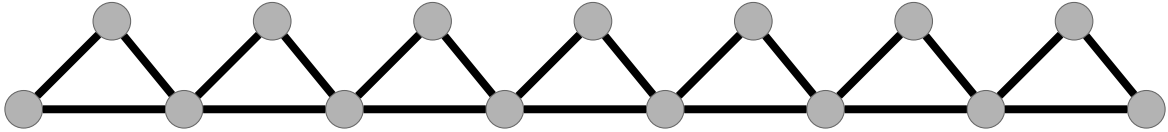
\begin{figure}[t!] 
\centering
    \begin{tikzpicture}[scale=0.5, every node/.style={transform shape, circle, draw=black!60, fill=black!30, minimum size=1cm, font=\bfseries}]
        \node (B) at (0,0) {};
        \node (A) [above right = 2.3cm of B] {};
        \node (C) [right = 3.25cm of B] {};
        \node (D) [above right = 2.3cm of C] {};
        \node (E) [right = 3.25cm of C] {};
        \node (F) [above right = 2.3cm of E] {};
        \node (G) [right = 3.25cm of E] {};
        \node (H) [above right = 2.3cm of G] {};
        \node (I) [right = 3.25cm of G] {};
        \node (J) [above right = 2.3cm of I] {};
        \node (K) [right = 3.25cm of I] {};
        \node (L) [above right = 2.3cm of K] {};
        \node (M) [right = 3.25cm of K] {};
        \node (N) [above right = 2.3cm of M] {};
        \node (O) [right = 3.25cm of M] {};
        \draw[very thick, line width = 3pt] (A) -- node[pos=0.5, sloped, above, font=\bfseries\Large, draw=none, fill=none, inner sep=0pt] {} (B); 
        \draw[very thick, line width = 3pt] (B) -- node[pos=0.5, sloped, above, font=\bfseries\Large, draw=none, fill=none, inner sep=0pt] {} (C);
        \draw[very thick, line width = 3pt] (C) -- node[pos=0.3, sloped, above, font=\bfseries\Large, draw=none, fill=none, inner sep=0pt] {} (D);
        \draw[very thick, line width = 3pt] (C) -- node[pos=0.5, sloped, above, font=\bfseries\Large, draw=none, fill=none, inner sep=0pt] {} (E); 
        \draw[very thick, line width = 3pt] (E) -- node[pos=0.5, sloped, above, font=\bfseries\Large, draw=none, fill=none, inner sep=0pt] {} (F); 
        \draw[very thick, line width = 3pt] (E) -- node[pos=0.5, sloped, above, font=\bfseries\Large, draw=none, fill=none, inner sep=0pt] {} (G); 
        \draw[very thick, line width = 3pt] (G) -- node[pos=0.5, sloped, above, font=\bfseries\Large, draw=none, fill=none, inner sep=0pt] {} (H); 
        \draw[very thick, line width = 3pt] (G) -- node[pos=0.5, sloped, above, font=\bfseries\Large, draw=none, fill=none, inner sep=0pt] {} (I); 
        \draw[very thick, line width = 3pt] (I) -- node[pos=0.5, sloped, above, font=\bfseries\Large, draw=none, fill=none, inner sep=0pt] {} (J); 
        \draw[very thick, line width = 3pt] (I) -- node[pos=0.5, sloped, above, font=\bfseries\Large, draw=none, fill=none, inner sep=0pt] {} (K); 
        \draw[very thick, line width = 3pt] (K) -- node[pos=0.5, sloped, above, font=\bfseries\Large, draw=none, fill=none, inner sep=0pt] {} (L); 
        \draw[very thick, line width = 3pt] (K) -- node[pos=0.5, sloped, above, font=\bfseries\Large, draw=none, fill=none, inner sep=0pt] {} (M); 
        \draw[very thick, line width = 3pt] (M) -- node[pos=0.5, sloped, above, font=\bfseries\Large, draw=none, fill=none, inner sep=0pt] {} (N); 
        \draw[very thick, line width = 3pt] (M) -- node[pos=0.5, sloped, above, font=\bfseries\Large, draw=none, fill=none, inner sep=0pt] {} (O);
        \draw[very thick, line width = 3pt] (C) -- node[pos=0.5, sloped, above, font=\bfseries\Large, draw=none, fill=none, inner sep=0pt] {} (A);
        \draw[very thick, line width = 3pt] (E) -- node[pos=0.5, sloped, above, font=\bfseries\Large, draw=none, fill=none, inner sep=0pt] {} (D);
        \draw[very thick, line width = 3pt] (G) -- node[pos=0.5, sloped, above, font=\bfseries\Large, draw=none, fill=none, inner sep=0pt] {} (F);
        \draw[very thick, line width = 3pt] (I) -- node[pos=0.5, sloped, above, font=\bfseries\Large, draw=none, fill=none, inner sep=0pt] {} (H);
        \draw[very thick, line width = 3pt] (K) -- node[pos=0.5, sloped, above, font=\bfseries\Large, draw=none, fill=none, inner sep=0pt] {} (J);
        \draw[very thick, line width = 3pt] (L) -- node[pos=0.5, sloped, above, font=\bfseries\Large, draw=none, fill=none, inner sep=0pt] {} (M);
        \draw[very thick, line width = 3pt] (O) -- node[pos=0.5, sloped, above, font=\bfseries\Large, draw=none, fill=none, inner sep=0pt] {} (N);
    \end{tikzpicture}
    \caption{Sawtooth graph} \label{fig:sawtooth}
\end{figure}







\section{Reductions for computing the metric geodesic cover number} \label{sec:theory}


Definition \ref{defi:mgcn} requires one to consider all possible geodesic covers of the space, which makes determining the cover number of a space \textit{a priori} a difficult problem. In this section, we prove Theorem \ref{thm:2subdivision}, which simplifies the picture by showing that we need only consider covers from a specific finite set of configurations. For a given cover, it is then routine to check (by computer) whether it is realizable as a collection of geodesics; one must check whether a set of linear constraints has a nonempty feasible region. 


We first introduce some terminology. Given a graph \(G\), the \textit{2-subdivision} of \(G\) is the graph \(G'\) formed by adding for each edge $(u,v)$ a new vertex $w$ and replacing \((u,v)\) with two edges \((u,w)\) and \((w,v)\).

Next, we say that a geodesic \(\alpha\) is \textit{proper} if there is no edge whose interior contains both endpoints of \(\alpha\). A geodesic cover of \(G\) is \textit{proper} if it consists only of proper geodesics. Given a geodesic cover $\Gamma$ of $G$, we say that a geodesic $\alpha \in \Gamma$ is \textit{retractable in $\Gamma$} if there is a neighborhood of an endpoint of $\alpha$ that is contained in the union of the other geodesics of $\Gamma$. Correspondingly, we say that $\Gamma$ is \textit{retracted} if it does not contain any retractable geodesics. It is easy to see that any geodesic cover $\Gamma$ can be made into a retracted geodesic cover by shortening (or removing altogether) any retractable geodesics as needed. 

\begin{figure}[t!]
\centering
\captionsetup[subfloat]{width=4cm, justification=centering}
\subfloat[Original geodesic cover]{
    \begin{tikzpicture}[scale=0.4, every node/.style={transform shape, circle, draw=black!60, fill=black!30, minimum size=1cm, font=\bfseries}]
        \useasboundingbox (-1,-7.5) rectangle (14,2);
        \node (A) {};
        \node (B) [below right = 2.5cm and 2.5cm of A] {};
        \node (C) [below = 5.5cm of A] {};
        \coordinate (D) at ($(B) + (2.5cm, 0cm)$);
        \node (E) [below right = 2.5cm and 9cm of A] {};
        \node (F) [right = 11.5cm of C] {};
        \draw[red!90!black, very thick] ($(D)+(0,0.5)$) -- ($(D)+(0,-0.5)$); 
        \draw [red!90!black, line width = 3pt] (A) -- (B);  
        \draw [red!90!black, line width = 6pt] (B) -- (D);
        \draw [blue!90!black, line width = 3pt] (C) -- (B);  
        \draw [blue!90!black, line width = 2.5pt] (B) -- (D);
        \draw [blue!90!black, line width = 3pt] (D) -- (E);
        \draw [blue!90!black, line width = 3pt] (E) -- (F); 
    \end{tikzpicture}
}
\hspace{2.5 cm}
\subfloat[Retracted geodesic cover]{
    \begin{tikzpicture}[scale=0.4, every node/.style={transform shape, circle, draw=black!60, fill=black!30, minimum size=1cm, font=\bfseries}]
        \useasboundingbox (-1,-7.5) rectangle (14,2);
        \node (A) {};
        \node (B) [below right = 2.5cm and 2.5cm of A] {};
        \node (C) [below = 5.5cm of A] {};
        \node (D) [below right = 2.5cm and 9cm of A] {};
        \node (E) [right = 11.5cm of C] {};
        \draw [red!90!black, line width = 3pt] (A) -- (B);  
        \draw [blue!90!black, line width = 3pt] (B) -- (C);  
        \draw [blue!90!black, line width = 3pt] (B) -- (D);
        \draw [blue!90!black, line width = 3pt] (D) -- (E);
    \end{tikzpicture}
}
\caption{Retracting a geodesic cover.}
\label{fig:geodesic-retraction}
\end{figure}
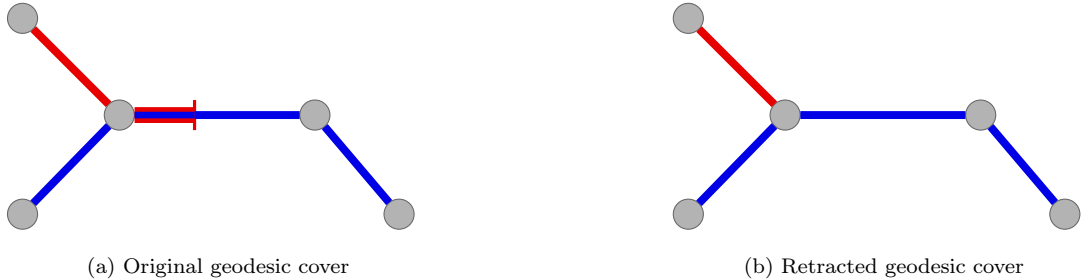




The proof of Theorem \ref{thm:2subdivision} relies on the following lemma.



\begin{lemm} \label{lemm:proper_cover}
    If \(G\) has a geodesic cover of size \(k\), then it has a proper retracted geodesic cover of size \(k\).
\end{lemm}

\begin{proof}
Let $\Gamma_0$ be a retracted geodesic cover of $G$. We apply the following procedure to replace, one by one, each nonproper geodesic in $\Gamma_0$ and its neighboring geodesics with the same number of proper geodesics covering the same subset of $G$. The set of non-proper geodesics can be split into two types, based on whether or not a geodesic is completely contained in the interior of its corresponding edge. These two possibilities are shown in Figure \ref{fig:improper-geodesics}. We apply the procedure to all nonproper geodesics of the first type; we will see that any geodesics of the second type are eliminated as a byproduct of this process. 

\begin{figure}[t!]
\centering
\captionsetup[subfloat]{width=4cm, justification=centering}
\subfloat[An improper geodesic contained in a single edge]{
    \begin{tikzpicture}[scale=0.4, every node/.style={transform shape, circle, draw=black!60, fill=black!30, minimum size=1cm, font=\bfseries}]
        \useasboundingbox (-1,-1) rectangle (11,3.5);
        \node (A) at (0,0) {};
        \coordinate (B) at (2.5,0);
        \coordinate (C) at (7.5,0);
        \node (D) at (10,0) {};
        \draw (A) -- (D);  
        \draw[red!90!black, very thick, line width=3pt] (B) -- (C); 
        \draw[red!90!black, very thick] ($(B)+(0,0.5)$) -- ($(B)+(0,-0.5)$); 
        \draw[red!90!black, very thick] ($(C)+(0,0.5)$) -- ($(C)+(0,-0.5)$); 
    \end{tikzpicture}
}
\hspace{3cm}
\subfloat[An improper geodesic crossing multiple edges]{
    \begin{tikzpicture}[scale=0.4, every node/.style={transform shape, circle, draw=black!60, fill=black!30, minimum size=1cm, font=\bfseries}]
        \useasboundingbox (-1,-1) rectangle (11,3.5);
        \node (A) at (0,0) {};
        \coordinate (B) at (2.5,0);
        \coordinate (C) at (7.5,0);
        \node (D) at (10,0) {};
        \draw (A) -- (D);  
        \draw[red!90!black, very thick, line width=3pt] (A) -- (B); 
        \draw[red!90!black, very thick, line width=3pt] (C) -- (D);  
        \draw[red!90!black, very thick] ($(B)+(0,0.5)$) -- ($(B)+(0,-0.5)$); 
        \draw[red!90!black, very thick] ($(C)+(0,0.5)$) -- ($(C)+(0,-0.5)$);
        \draw[red!90!black, very thick, line width=3pt, bend left=85] (A) to (D);
    \end{tikzpicture}
}
\caption{Examples of improper geodesics.}
\label{fig:improper-geodesics}
\end{figure}
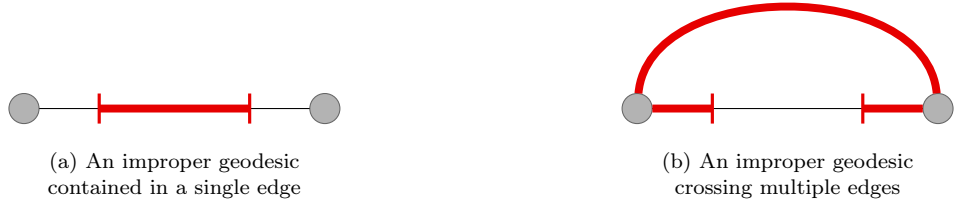


Let $\Gamma$ be a geodesic cover of $G$ appearing in some stage of this process. Let \(\alpha \in \Gamma\) be a nonproper geodesic. Assume that $\alpha$ is completely contained within the interior of the corresponding edge \(e\). We assume that $e$ is not a loop, i.e., the two endpoints of $e$ are distinct. The case where $e$ is a loop is handled by a similar but simpler argument. Fix compatible orientations on $e$ and $\alpha$; let $u$ be the initial point and $v$ the final point of $e$, and let $x$ be the initial point and $y$ the final point of $\alpha$. 

Now let $\alpha'$ be the maximal geodesic contained in $e$ having initial point $u$. Let $y'$ denote the other endpoint of $\alpha'$. Next, let $\beta$ be any geodesic in $\Gamma$ containing $v$ and intersecting the interior of $e$. Note that $\beta$ must have an endpoint $w$ in the interior of $e$. Form a new path from $\beta$ in the following way. First, we extend or shorten the endpoint $w$ to the point $y'$. Next, if the other endpoint of $\beta$ also lies in the interior of $e$, then retract this endpoint to $u$. Denote the resulting path by $\beta'$. See Figure \ref{fig:proper_cover} for an illustration of this configuration. Form the new collection of paths $\Gamma'$ by replacing $\alpha$ and $\beta$ with the paths $\alpha'$ and $\beta'$, and then retracting the other geodesics in $\Gamma$ as needed. Note that these retraction are of two types: retracting a geodesic crossing through $u$ with endpoint in $e$ back onto $u$, or removing any other nonproper geodesics contained in the edge $e$. These operations do not create any new nonproper geodesics. Observe that $\Gamma'$ is still a cover of $G$ of the same size as $\Gamma$, and $\alpha'$ by definition is a geodesic. Moreover, both $\alpha'$ and $\beta'$ are proper. It remains to show that $\beta'$ is a geodesic.

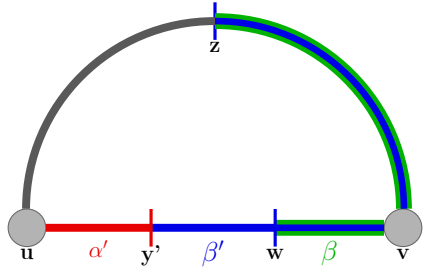
\begin{figure}[t!]
\centering
    \begin{tikzpicture}[scale=0.5, every node/.style={transform shape, circle, draw=black!60, fill=black!30, minimum size=1cm, font=\bfseries}]
        \node (A) at (0,0) {};
        \node[below=5pt, font=\bfseries\huge, fill=none, draw=none] at (A) {u};
        \coordinate (B) at (3.3,0);
        \node[below=5pt, font=\bfseries\huge, fill=none, draw=none] at (B) {y'};
        \coordinate (C) at (6.6,0);
        \node[below=5pt, font=\bfseries\huge, fill=none, draw=none] at (C) {w};
        \node (D) at (10,0) {};
        \node[below=5pt, font=\bfseries\huge, fill=none, draw=none] at (D) {v};
        \coordinate (E) at (5,5.5);
        \node[below=5pt, font=\bfseries\huge, fill=none, draw=none] at (E) {z};
        \draw (A) -- (D);  
        \draw[red!90!black, very thick, line width = 3pt] (A) -- (B)node[midway, below, draw=none, fill=none, font=\huge] {$\alpha'$}; 
        \draw[green!70!black, very thick, line width = 6pt] (C) -- (D)
        node[midway, below, draw=none, fill=none, font=\bfseries\Huge,] {$\beta$};
        \draw[blue!90!black, very thick, line width = 2.5pt] (C) -- (D);  
        \draw[blue!90!black, very thick, line width = 3pt] (B) -- (C) node[midway, below, draw=none, fill=none, font=\bfseries\Huge,] {$\beta'$};  
        \draw[red!90!black, very thick] ($(B)+(0,0.5)$) -- ($(B)+(0,-0.5)$); 
        \draw[blue!90!black, very thick] ($(C)+(0,0.5)$) -- ($(C)+(0,-0.5)$);
        \draw[blue!90!black, very thick] ($(E)+(0,0.5)$) -- ($(E)+(0,-0.5)$);
        \draw[gray!70!black, very thick, line width = 3pt, bend left = 45] (A) to (E);
        \draw[green!70!black, very thick, line width = 6pt, bend left = 45] (E) to (D);
        \draw[blue!90!black, very thick, line width = 2.5pt, bend left = 45] (E) to (D);
    \end{tikzpicture}
\caption{Configuration of geodesics in the proof of Lemma \ref{lemm:proper_cover}} \label{fig:proper_cover}
\end{figure}

If $w$ lies in $\alpha'$, then $\beta'$ is contained in $\beta$ and hence is a geodesic. So we may assume that $w$ is not contained in $\alpha'$. Let \(z\) be the endpoint of \(\beta'\) not within the interior of \(e\). Let $\gamma$ be an arbitrary path from $y'$ to $z$. If $\gamma$ passes through the point $v$, then it is immediate that $\ell(\gamma) \geq \ell(\beta')$ since the restriction of $\beta'$ from $v$ to $z$ is a geodesic, and the two paths coincide prior to reaching $v$.


Assume now that $\gamma$ passes through the point $u$. The maximality of $\alpha'$ implies that $y'$ and $u$ are joined by a second geodesic $\sigma$ that passes through $v$. Let $\gamma'$ be the concatenation of $\sigma$ and $\gamma|_{[u,z]}$; it is immediate that $\gamma$ and $\gamma'$ have the same length. But $\gamma'$ passes through $v$, so as above we have $\ell(\gamma) = \ell(\gamma') \geq \ell(\beta')$. We conclude that $\beta'$ is indeed a geodesic.

We now claim that, following this process, the resulting cover $\Gamma$ does not contain any nonproper geodesics of the second type. Consider now an edge $e = (u,v)$ in $G$. Since $\Gamma$ is retracted, if $e$ is entirely contained in a single geodesic, then no other geodesic may have an endpoint in the interior of $e$. Otherwise, there is a geodesic $\gamma_1$ through $u$ that has endpoint in $e$ closest to $v$ and a geodesic $\gamma_2$ through $v$ that has endpoint in $e$ closest to $u$. Since we have eliminated any nonproper geodesics contained in $e$, it follows that $\gamma_1$ and $\gamma_2$ together cover $e$. Since $\Gamma$ is retracted, $\gamma_1$ and $\gamma_2$ must share the same endpoint, and in particular it follows that $\gamma_1 \neq \gamma_2$. Moreover, no other geodesic may intersect the interior of $e$. The claim follows. 
\end{proof}

We continue with the proof of Theorem \ref{thm:2subdivision}.
\begin{proof}
    Let $\Gamma$ be an arbitrary geodesic cover of $G$. By Lemma \ref{lemm:proper_cover}, we can assume that $\Gamma$ is proper and retracted. In particular, for each edge $e =(u,v)$ one of the following holds: (1) $e$ is contained in a single geodesic and no other geodesic has an endpoint in the interior of $e$, or (2) there are two geodesics with common endpoint $x$ in the interior of $e$, with one geodesic passing through $u$ and the other passing through $v$. In the second case, we can identify $x$ with the vertex of the $2$-subdivision of $G$ in the interior of $e$. 
\end{proof}

    \section{Spaces that can be covered by three geodesics} \label{sec:three_geodesics}
    
    The case of three geodesics is simple enough that there is an explicit description of spaces with covering number three. This is done in Proposition \ref{prop:3geodesic_catalogue} below.

    An \textit{orientation} on a simple path $Y$ is a choice of initial point and end point; an orientation induces of an order on $Y$, denoted by $\prec$, where $x \prec y$ if, when traveling from the initial point of \(Y\) to its end point, we encounter \(x\) before \(y\). Two oriented paths $Y_1, Y_2$, with respective orders $\prec_1$ and $\prec_2$, are said to be \textit{compatibly oriented} if for all $x,y \in Y_1 \cap Y_2$, $x \prec_1 y$ if and only if $x \prec_2 y$. In this case, the two orders $\prec_1$ and $\prec_2$ induce a partial order on $Y_1 \cup Y_2$ compatible with the respective orders on $Y_1$ and $Y_2$. If $Y$ is a simple path containing the points $x,y$, we let $Y(x,y)$ denote the subpath of $Y$ between $x$ and $y$.  

    As a warmup, we have the following proposition concerning two geodesics. 

    \begin{prop} \label{prop:2geodesics}
        Let $X$ be a connected topological space that is the union of two subsets $X_1, X_2$ each homeomorphic to an interval. Then $X$ admits a length metric in which $X_1, X_2$ are geodesics if and only if $X_1, X_2$ admit compatible orientations.
    \end{prop}
    \begin{proof}
        Assume first that $X$ admits a length metric in which $X_1, X_2$ are geodesics. If $X_1, X_2$ intersect in a single point, then any orientations are trivially compatible. Assume then that $X_1$ and $X_2$ intersect in at least two points. Choose an orientation on $X_1$, with $\prec_1$ the corresponding order. Let $a$ be the first point in $X_1 \cap X_2$ according to this orientation and $b$ the last point in $X_1 \cap X_2$. We obtain an orientation, with corresponding order $\prec_2$, on $X_2$ by declaring $a \prec_2 b$.

        Suppose that there exist points $x,y \in X_1 \cap X_2$ such that $x \prec_1 y$ but $y \prec_2 x$. Note that, by the definition of $a$ and $b$, it necessarily holds that $a \preceq_1 x \prec_1 y \preceq_1 b$. If $y \prec_2 a$, then we would get the contradiction
        \[d(y,b) = \ell(X_2(y,b)) > \ell(X_2(a,b)) = \ell(X_1(a,b)) > \ell(X_1(y,b)) = d(y,b). \]
        Thus we see that $a \preceq_2 y$. We next immediately get the contradiction
        \[d(a,y) = \ell(X_2(a,y)) < \ell(X_2(a,x)) = \ell(X_1(a,x)) < \ell(X_1(a,y)) = d(a,y). \]
        We conclude that any two points $x,y$ satisfy $x \prec_1 y$ if and only if $x \prec_2 y$, and so $X_1, X_2$ have compatible orientations.

        For the reverse implication, assume that $X_1, X_2$ admit compatible orientations $\prec_1, \prec_2$. Denote the endpoints of $X_i$ by $a_i,b_i$ for each $i=1,2$. If $X_1 \cap X_2$ contains zero or one points, then any choice of length metric works. So assume $X_1 \cap X_2$ contains at least two points. Fix a homeomorphism $\varphi \colon X_1 \to [0,1]$ such that $\varphi^{-1}(0) \prec_1 \varphi^{-1}(1)$. We extend the definition of $\varphi$ to all of $X$ as follows. Each component of $X_2 \setminus X_1$ is an open or half-open interval. Let $I$ be a component that is an open interval, with endpoints $a,b \in X_1 \cap X_2$. Define $\varphi$ on $I$ to be an arbitrary embedding into $[0,1]$ with $\lim_{x \to a} \varphi(x) =\varphi(a)$ and $\lim_{x \to b} \varphi(x) =\varphi(b)$. We define $\varphi$ similarly if $I$ is a half-closed interval, with the requirement that $\varphi(a_2) =0$ and $\varphi(b_2) = 1$. We now let $d$ be the length metric on $X$ induced by the map $\varphi$. That is, we define the length of a path $\gamma$ as
        \begin{equation} \label{eq:length}
            \ell(\gamma) = \sup \sum_{i=1}^n |\varphi(\gamma(t_{i-1})) - \varphi(\gamma(t_i)))|,
        \end{equation}
        the supremum taken over all finite partitions $t_0 < t_1 < \ldots < t_n$ of the domain of $\gamma$. We then $d(x,y) = \inf \ell(\gamma)$, the infimum taken over all paths $\gamma$ from $x$ to $y$. It is immediate from construction that $d$ is a length metric on $X$.

        It remains to show that $X_1$ and $X_2$ are geodesics in this metric. Observe that $X_1$ itself has length $1$ by construction. Now let $\gamma$ be an arbitrary path from $a_1$ to $b_1$. Now $\varphi(\gamma)$ is a surjection onto $[0,1]$, from which it follows immediately from \eqref{eq:length} and the triangle inequality that $\ell(\gamma) \geq 1$. We conclude that $\ell(X_1) = 1 = d(a_1,b_1)$, and hence that $X_1$ is a geodesic. 

        Next, we claim that $X_2$ has length $\varphi(b_2) - \varphi(a_2)$. Let $t_0 < t_1 < \cdots < t_n$ be an arbitrary partition of the domain of $X_2$. It follows from the orientation assumption that $\varphi(X_2(t_{i-1})) < \varphi(X_2(t_{i}))$ for all $i$. Thus we see from \eqref{eq:length} that $X_2$ has length $\varphi(b_2) - \varphi(a_2)$. Finally, if $\gamma$ be an arbitrary path from $a_2$ to $b_2$, the same argument as in the last paragraph shows that $\ell(\gamma) \geq \varphi(b_2) - \varphi(a_2)$. We conclude that $\ell(X_2) = d(a_2,b_2)$, and hence that $X_2$ is a geodesic.  
    \end{proof}

    The next proposition gives a description of topological spaces that are the union of three geodesics. First, we state some terminology. Let $X_1, X_2$ be two paths in a topological space $X$. We say a component $U$ of $X_1 \setminus X_2$ is an \textit{end component} if it is homeomorphic to half-open interval $[0,1)$. Note that the image of $0$ in $X_1$ under the homeomorphism is necessarily an endpoint of $X_1$. By convention, we allow the end components to be empty, so that we can say they always exist. Next, assume that $U$ is a component of $X_1 \setminus X_2$ and $V$ is a component of $X_2 \setminus X_1$, where $X_1$ and $X_2$ are compatibly oriented. We say that $U$ and $V$ are \textit{comparable} if every point in one set is greater than or equal to every point in the other according to induced partial order on $X_1 \cup X_2$.
 
    \begin{prop} \label{prop:3geodesic_catalogue}
        Let $X$ be a connected topological space that is the union of three subsets $X_1, X_2, X_3$ each homeomorphic to an interval. Then $X$ admits a length metric in which $X_1, X_2, X_3$ are each geodesics if and only if $X_1, X_2, X_3$ can be given orientations such that one of the following holds (after relabeling):
        \begin{enumerate}
            \item The orientations on $X_1, X_2, X_3$ induce a compatible partial order on $X$.
            \item  The space $X$ has one of two possible exceptional configurations as shown in Figure \ref{fig:exceptional}. \label{item:non_oriented}
            \begin{enumerate}
                \item Each pair $X_i, X_j$ ($i \neq j$) is mutually compatibly oriented. Let $U$ be the last end component of $X_1 \setminus X_2$ and $V$ the first end component of $X_2 \setminus X_1$. Then $X_1 \cap X_3 \subset U$ and $X_2 \cap X_3 \subset V$.
                \item The pairs $X_1, X_2$ and $X_1, X_3$ are compatibly oriented, while the pair $X_2,X_3$ is not. There are noncomparable components $U$ of $X_1 \setminus X_2$ and $V$ of $X_2 \setminus X_1$ with common last endpoint $e$ such that $X_1 \cap X_3 \subset U \cup \{e\}$ and $X_2 \cap X_3 \subset V \cup \{e\}$.
            \end{enumerate} 
        \end{enumerate} 
    \end{prop} 

    \begin{figure}[t!] 
    \centering
    \hfill 
    \subfloat[]{
    \begin{tikzpicture}[scale=.5] 
    \draw[purple, line width = 11pt, opacity = .2] (0,-.0) to (0,4.);
    \draw[red, line width = 11pt, opacity = .2] (.4,-.4) to (2.5,-2.5);
    \draw[blue, line width = 11pt, opacity = .2] (-.4,4.4) to (-2.5,6.5);
    \draw[green!70!black, line width = 2.5pt] (1.9,1.1) to (.4,-.4) .. controls (.25, -.95) and (.55, -1.25) .. (1.1,-1.1) .. controls (.95,-1.65) and (1.25, -1.95) .. (1.8, -1.8) .. controls (1.65, -2.35) and (1.95, -2.65) .. (2.5, -2.5) .. controls (3.75,-2.2) and (4,-1.5) .. (4, 0) to (4,7) .. controls (4,7.6) and (3.6,8) .. (3,8) to (0,8) .. controls (-1.5, 8) and (-1.5, 7.5) .. (-2.5, 6.5) .. controls (-1.95,6.65) and (-1.65,6.35) .. (-1.8, 5.8).. controls (-1.25, 5.95) and (-.95, 5.65) .. (-1.1, 5.1) .. controls (-.55, 5.25) and (-.25, 4.95) .. (-.4, 4.4) to (-1.9,2.9) ; 
    \draw[blue!90!black, line width = 2.5pt] (-3,-3) to (0,0) .. controls (-.5,.29) and (-.5,.71) .. (0,1) .. controls (-.5,1.29) and (-.5,1.71) .. (0,2) .. controls (-.5,2.29) and (-.5,2.71) .. (0,3) .. controls (-.5,3.29) and (-.5,3.71) .. (0,4) to (-3,7);
    \draw[red!90!black, line width = 2.5pt] (3,-3) to (0,0) .. controls (.5,.29) and (.5,.71) .. (0,1) .. controls (.5,1.29) and (.5,1.71) .. (0,2) .. controls (.5,2.29) and (.5,2.71) .. (0,3) .. controls (.5,3.29) and (.5,3.71) .. (0,4) to (3,7);
    \draw[draw=black!60, fill=black!30] (0,0) circle (6pt);
    \draw[draw=black!60, fill=black!30] (.4,-.4) circle (6pt);
    \draw[draw=black!60, fill=black!30] (2.5,-2.5) circle (6pt);
    \draw[draw=black!60, fill=black!30] (0,4) circle (6pt);
    \draw[draw=black!60, fill=black!30] (-.4,4.4) circle (6pt);
    \draw[draw=black!60, fill=black!30] (-2.5,6.5) circle (6pt);
    \end{tikzpicture}
    \label{fig:exceptional_1}}
    \hfill 
    \subfloat[]{
    \begin{tikzpicture}[scale=.5] 
    \draw[purple, line width = 11pt, opacity = .2] (0,-1) to (0,2.);
    \draw[purple, line width = 11pt, opacity = .2] (0,7) to (0,10);
    \draw[red, line width = 11pt, opacity = .2] (0,7) to (2.2,4.8);
    \draw[blue, line width = 11pt, opacity = .2] (0,7) to (-2.2,4.8);
     \draw[green!70!black, line width = 2.5pt] (-3.7,4.8) to (-2.2, 4.8) .. controls (-1.65, 4.65) and (-1.35, 4.95) .. (-1.5, 5.5) .. controls (-.95, 5.35) and (-.65,5.65) .. (-.8,6.2) .. controls (-.25, 6.05) and (-.05, 6.25) .. (0,7)  .. controls (.05, 6.25) and (.25, 6.05) .. (.8,6.2) .. controls (.65,5.65) and (.95, 5.35) .. (1.5, 5.5) .. controls (1.35, 4.95) and (1.65, 4.65) .. (2.2, 4.8) to (3.7,4.8);
     \draw[blue!90!black, line width = 2.5pt] (-1,-2) to (0,-1) .. controls (-.5,-.71) and (-.5,-.29) .. (0,0) .. controls (-.5,.29) and (-.5,.71) .. (0,1) .. controls (-.5,1.29) and (-.5,1.71) .. (0,2) to (-2,4) .. controls (-2.3,4.3) and (-2.3,4.7) .. (-2,5) to (0,7) .. controls (-.5,7.29) and (-.5,7.71) .. (0,8) .. controls (-.5,8.29) and (-.5,8.71) .. (0,9) .. controls (-.5,9.29) and (-.5,9.71) .. (0,10) to (-1,11);
    \draw[red!90!black, line width = 2.5pt] (1,-2) to (0,-1) .. controls (.5,-.71) and (.5,-.29) .. (0,0) .. controls (.5,.29) and (.5,.71) .. (0,1) .. controls (.5,1.29) and (.5,1.71) .. (0,2) to (2,4) .. controls (2.3,4.3) and (2.3,4.7) .. (2,5) to (0,7) .. controls (.5,7.29) and (.5,7.71) .. (0,8) .. controls (.5,8.29) and (.5,8.71) .. (0,9) .. controls (.5,9.29) and (.5,9.71) .. (0,10) to (1,11);
    \draw[draw=black!60, fill=black!30] (0,-1) circle (6pt);
    \draw[draw=black!60, fill=black!30] (0,2) circle (6pt);
    \draw[draw=black!60, fill=black!30] (2.2,4.8) circle (6pt);
    \draw[draw=black!60, fill=black!30] (0,7) circle (6pt);
    \draw[draw=black!60, fill=black!30] (-2.2,4.8) circle (6pt);
    \draw[draw=black!60, fill=black!30] (0,10) circle (6pt);
    \end{tikzpicture}
    \label{fig:exceptional_2}}
    \hfill \text{ }
    \caption{The geodesics are constrained to intersect only within the indicated intervals and with the indicated orientations.} \label{fig:exceptional}
\end{figure}
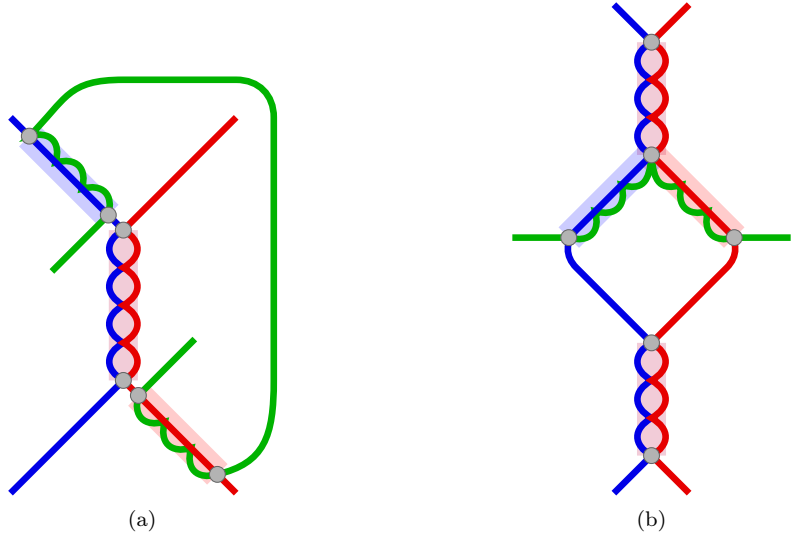

    This main thrust of the proof is to consider all possible ways that the paths $X_1, X_2, X_3$ can intersect and rule out any configurations incompatible with being geodesics. Since there is a large number of possible configurations, the bulk of this analysis is deferred until Appendix B. 

    \begin{proof}
    Let $X$ be a topological space that is the union of three paths $X_1, X_2, X_3$. We first observe that if $X_1, X_2, X_3$ can be given orientations inducing a partial order on $X$, then $X$ carries a length metric in which $X_1, X_2, X_3$ are geodesics. This follows from a similar argument as in Proposition \ref{prop:2geodesics}. 
    
    It is easy to see that if $|X_i \cap X_j| \leq 1$ and $|X_i \cap X_k| \leq 1$ for some distinct $i,j,k \in \{1,2,3\}$, then $X$ can be given a compatible partial order. Thus, in the following, we will assume that $|X_1 \cap X_2| \geq 2$ and $|X_1 \cap X_3| \geq 2$. Choose an orientation on $X_1$; this induces orientations on $X_2$ and $X_3$. As noted, if these three orientations are compatible with a partial order on $X$, then $X$ can be given a metric in which $X_1, X_2, X_3$ are geodesics. There are two ways that this can fail (after reversing the orientation on $X_1$ if needed): (1) there exist points $a \in X_1 \cap X_3$, $b \in X_1 \cap X_2$ and $c \in X_2 \cap X_3$ such that $a \prec_1 b$, $b \prec_2 c$ and $c \prec_3 a$, and (2) there exist points $a,b \in X_2 \cap X_3$ such that $a \prec_3 b$ and $b \prec_2 a$. An analysis of these cases is carried out in Appendix B, which rules out all configurations except those listed in item \eqref{item:non_oriented} of the theorem.

    To finish the proof, we show that there is a metric realizing these remaining configurations. 

    \begin{enumerate}
    \item[(a)] Let $I_{ij}$ be the portion of $X_i \cup X_j$ bracketed by $X_i \cap X_j$ (the regions between the gray points in Figure \ref{fig:exceptional_1}). Assign each $I_{ij}$ a length of $1$, distributed as was done in the proof of Proposition \ref{prop:2geodesics}. Assign each subarc of $X_1, X_2, X_3$ between the $I_{ij}$ a length of $1$. Assign the length of the remaining edges arbitrarily.
    \item[(b)] Define $I_{12}^a, I_{12}^b,I_{13}, I_{23}$ similarly to the previous case. Assign each a length of $1$. Assign the edges in between a length of $1$, and the other edges arbitrarily.
    \end{enumerate} 

    In each case, one can check that $X_1, X_2, X_3$ are geodesics in the given metric.
    \end{proof} 
 
    We are now able to prove Theorem \ref{thm:planar} that all topological spaces $X$ with geodesic cover number at most $3$ are planar. We note for the proof that it is not enough to rule out the graphs $K_5$ and $K_{3,3}$ appearing as subsets of $X$. Since $X$ is not assumed to be a topological graph, there are other configurations that are forbidden from appearing as a subset; see Claytor \cite{Clay:37}. Rather than deal directly with these possibilities, our proof is simply to construct the required embedding.  

    \begin{proof}
    First, we observe that a planar embedding is easy to find for the case of two geodesics. Consider two compatibly oriented overlapping geodesics $X_1$ and $X_2$ parameterized by arc length, where $X_1$ is defined on $[a_1,b_1]$ and $X_2$ is defined on $[a_2,b_2]$, with $X_1^{-1}(x) = X_2^{-1}(x)$ for all $x \in X_1 \cap X_2$. Then define a map $\varphi$ on $X_1$ by setting $\varphi(X_1(t)) = (t,\pm d(X_1(t), X_2))$ and $\varphi(X_2(t)) = (t, \mp d(X_2(t), X_1))$, where the choice of $\pm$ is consistent on each component of $X_1 \setminus X_2$ and the choice of $\mp$ on the noncomparable component of $X_2 \setminus X_1$ is made to be opposite. It is easy to see that this is a well-defined continuous, injective, closed map and hence an embedding of $X_1 \cup X_2$ into $\mathbb{R}^2$.

    Next, suppose that $X$ is the union of three geodesics $X_1, X_2, X_3$. If $X$ has one of the exceptional configurations in Figure \ref{fig:exceptional}, then we can embed $X$ into the plane by following the representations in these pictures. Namely, $X$ can be separated into components separated by the indicated gray vertices. Define $\varphi$ on these vertices by sending them to points in the plane with the same positioning as in the figure. Define $\varphi$ on each region of intersection of two geodesics by taking the affine image of an embedding as in the previous paragraph. Define $\varphi$ on the remaining segments by connecting these vertices with an arc. This can be done so that $\varphi$ is continuous, injective and closed and hence an embedding. 

    The final case is when $X_1, X_2, X_3$ induce a compatible partial order on $X$. In this case, embed first $X_1 \cup X_2$  as in the first paragraph, where the choice of $\pm$ on each component of $X_1 \setminus X_2$ will be determined in this proof. Enumerate the components of $X_3 \setminus (X_1 \cap X_2)$ by $U_1, U_2, U_3, \ldots$. For a given $U_i$, we next enumerate the components of $U_i \setminus (X_1 \cup X_2)$ as $U_i^1, U_i^2, \ldots$. The components $U_i^j$ have one of three possible types: (1) the endpoints of $U_i^j$ belong to the same component of $X_1 \setminus X_2$ or $X_2 \setminus X_1$; (2) one endpoint of $U_i^j$ belongs to a component $W_1$ of $X_1 \setminus X_2$ and the other endpoint belongs to a component $W_2$ of $X_2 \setminus X_1$, where the two components are noncomparable; (3) the endpoints of $U_i^j$ belong to components of $X_1 \setminus X_2$ or $X_2 \setminus X_1$ that are comparable. We make the choice of $\pm$ for each component of $X_1 \setminus X_2$ to ensure that, for each arc of type (3), the two endpoints belong to either the upper half-plane or lower half-plane. To be more precise, let $j_0$ be the smallest value such that $U_i^{j_0}$ is of type (3). Choose $\pm$ arbitrarily for the corresponding component of $X_1 \setminus X_2$. From this choice, the remaining components of type (3) can be enumerated as $U_i^{j_k}$ (in increasing order from $U_i^{j_0}$) and $U_i^{j_{-k}}$ (in decreasing order from $U_i^{j_0}$). The choice of $\pm$ is made inductively in both directions from this initial choice.  The direction of the remaining components of $X_1 \setminus X_2$ can be assigned arbitrarily. We now define $\varphi$ on each set $U_i^j$ either by a straight line between its endpoints (in the case of (2)) or by joining a segment of slope $2$ and one of slope $-2$ that connect the endpoints on the outside of $\varphi(X_1 \cup X_2)$. Note that the existence of a partial order on $X$ guarantees that each vertical line intersects at most one set $\varphi(U_i^j)$, and in particular that the sets $\varphi(U_i^j)$ do not intersect. We conclude that $\varphi$ is continuous, injective and closed and hence an embedding.
    \end{proof} 

    A similar description of spaces with cover number $n$ for $n \geq 5$ is not possible. For example, the well-known Petersen graph can be covered by $5$ paths with each pair intersecting a single point, yet not admitting any compatible length metric. See Figure 1a in \cite{cizma_geodesic_2022} for an illustration. Thus whether or not a space covered by $5$ paths admits a compatible metric is not determined by the set on which these paths intersect. We are not sure whether the case $n=4$ can be characterized in this way.
    
    \section{Computational results} \label{sec:computation}

    We this section, we present our computational results on a number of graphs of interest, found using the algorithm described in Appendix A.


    \subsection{The complete graph $K_5$ on five vertices } \label{sec:K5}

    The graph $K_5$ has cover number $4$. There are three distinct retracted geodesic covers attaining this value, as shown in Figure \ref{fig:K5}. Since $K_5$ is non-planar, it must have extended cover number $4$ by Theorem \ref{thm:planar}. It is easy to see that the unweighted cover number of $K_5$ is $5$, since $K_5$ contains $10$ edges and has diameter $2$ in the unweighted metric. All optimal geodesic covers of \(K_5\) are induced by a weighting where all edges in some \(K_4\) subgraph are length $2$, whereas the edges to some other ``shortcut'' vertex have length $1$. 

    \begin{figure}[t!]
    \centering
    \subfloat[]{
        \begin{tikzpicture}[scale=0.35, every node/.style={transform shape, circle, draw=black!60, fill=black!30, minimum size=1cm, font=\bfseries}]
            \node (A) at (0,0) {};
            \node (B) [below left = 5.775cm and 0.825cm of A] {};
            \node (C) [below right = 5.775cm and 0.825cm of A] {};
            \node (D) [below left = 10.45cm and 6.05cm of A] {};
            \node (E) [below right = 10.45cm and 6.05cm of A] {};
            \draw (A) -- (B);
            \draw (A) -- (C);
            \draw (A) -- (D);
            \draw (A) -- (E);
            \draw (B) -- (D);
            \draw (B) -- (C); 
            \draw (B) -- (E); 
            \draw (C) -- (E);
            \draw (C) -- (D); 
            \draw (D) -- (E); 
            \path (B) -- (D) coordinate[pos=0.5] (BDmid);
            \draw[red!90!black, very thick] ($(BDmid)+(0,0.5)$) -- ($(BDmid)+(0,-0.5)$);
            \path (B) -- (E) coordinate[pos=0.5] (BEmid);
            \draw[yellow!70!black, very thick] ($(BEmid)+(0,0.5)$) -- ($(BEmid)+(0,-0.5)$);
            \path (C) -- (D) coordinate[pos=0.5] (CDmid);
            \draw[blue!90!black, very thick] ($(CDmid)+(0,0.5)$) -- ($(CDmid)+(0,-0.5)$);
            \path (C) -- (E) coordinate[pos=0.5] (CEmid);
            \draw[green!70!black, very thick] ($(CEmid)+(0,0.5)$) -- ($(CEmid)+(0,-0.5)$);
            \draw[red!90!black, very thick, line width = 3pt] (BDmid) -- node[pos=0.5, sloped, above, font=\bfseries\Huge, draw=none, fill=none, inner sep=0pt] {1} (B); 
            \draw[red!90!black, very thick, line width = 3pt] (B) -- node[pos=0.4, sloped, above, font=\bfseries\Huge, draw=none, fill=none, inner sep=0pt] {1} (A); 
            \draw[red!90!black, very thick, line width = 3pt] (A) -- node[pos=0.6, sloped, above, font=\bfseries\Huge, draw=none, fill=none, inner sep=0pt] {1} (C); 
            \draw[red!90!black, very thick, line width = 3pt] (C) -- node[pos=0.5, sloped, above, font=\bfseries\Huge, draw=none, fill=none, inner sep=0pt] {1} (CEmid); 
            \draw[blue!90!black, very thick, line width = 3pt] (CDmid) -- node[pos=0.5, sloped, below, font=\bfseries\Huge, draw=none, fill=none, inner sep=0pt] {1} (D); 
            \draw[blue!90!black, very thick, line width = 3pt] (D) -- node[pos=0.5, sloped, above, font=\bfseries\Huge, draw=none, fill=none, inner sep=0pt] {1} (A); 
            \draw[blue!90!black, very thick, line width = 3pt] (A) -- node[pos=0.5, sloped, above, font=\bfseries\Huge, draw=none, fill=none, inner sep=0pt] {1} (E); 
            \draw[blue!90!black, very thick, line width = 3pt] (E) -- node[pos=0.5, sloped, below, font=\bfseries\Huge, draw=none, fill=none, inner sep=0pt] {1} (BEmid); 
            \draw[green!70!black, very thick, line width = 3pt] (BDmid) -- node[pos=0.1, sloped, above, font=\bfseries\Huge, draw=none, fill=none, inner sep=0pt] {1} (D); 
            \draw[green!70!black, very thick, line width = 3pt] (D) -- node[pos=0.5, sloped, below, font=\bfseries\Huge, draw=none, fill=none, inner sep=0pt] {2} (E); 
            \draw[green!70!black, very thick, line width = 3pt] (D) -- node[pos=0.5, sloped, below, font=\bfseries\Huge, draw=none, fill=none, inner sep=0pt] {2} (E); 
            \draw[green!70!black, very thick, line width = 3pt] (E) -- node[pos=0.9, sloped, above, font=\bfseries\Huge, draw=none, fill=none, inner sep=0pt] {1} (CEmid); 
            \draw[yellow!70!black, very thick, line width = 3pt] (BEmid) -- node[pos=0.3, sloped, below, font=\bfseries\Huge, draw=none, fill=none, inner sep=0pt] {1} (B); 
            \draw[yellow!70!black, very thick, line width = 3pt] (B) -- node[pos=0.5, sloped, above, font=\bfseries\Huge, draw=none, fill=none, inner sep=0pt] {2} (C); 
            \draw[yellow!70!black, very thick, line width = 3pt] (C) -- node[pos=0.7, sloped, below, font=\bfseries\Huge, draw=none, fill=none, inner sep=0pt] {1} (CDmid); 
        \end{tikzpicture}
    }
    \hfill 
    \subfloat[]{
        \begin{tikzpicture}[scale=0.35, every node/.style={transform shape, circle, draw=black!60, fill=black!30, minimum size=1cm, font=\bfseries}]
            \node (A) at (0,0) {};
            \node (B) [below left = 5.775cm and 0.825cm of A] {};
            \node (C) [below right = 5.775cm and 0.825cm of A] {};
            \node (D) [below left = 10.45cm and 6.05cm of A] {};
            \node (E) [below right = 10.45cm and 6.05cm of A] {};   
            \draw (A) -- (B);
            \draw (A) -- (C);
            \draw (A) -- (D);
            \draw (A) -- (E);
            \draw (B) -- (D);
            \draw (B) -- (C); 
            \draw (B) -- (E); 
            \draw (C) -- (E);
            \draw (C) -- (D); 
            \draw (D) -- (E); 
            \path (B) -- (C) coordinate[pos=0.5] (BCmid);
            \draw[red!90!black, very thick] ($(BCmid)+(0,0.5)$) -- ($(BCmid)+(0,-0.5)$);
            \path (B) -- (E) coordinate[pos=0.5] (BEmid);
            \draw[green!70!black, very thick] ($(BEmid)+(0,0.5)$) -- ($(BEmid)+(0,-0.5)$);
            \path (C) -- (D) coordinate[pos=0.5] (CDmid);
            \draw[yellow!70!black, very thick] ($(CDmid)+(0,0.5)$) -- ($(CDmid)+(0,-0.5)$);
            \path (D) -- (E) coordinate[pos=0.5] (DEmid);
            \draw[blue!90!black, very thick] ($(DEmid)+(0,0.5)$) -- ($(DEmid)+(0,-0.5)$);
            \draw[red!90!black, very thick, line width = 3pt] (BCmid) -- node[pos=0.5, sloped, above, font=\bfseries\Huge, draw=none, fill=none, inner sep=0pt] {1} (B); 
            \draw[red!90!black, very thick, line width = 3pt] (B) -- node[pos=0.4, sloped, above, font=\bfseries\Huge, draw=none, fill=none, inner sep=0pt] {1} (A); 
            \draw[red!90!black, very thick, line width = 3pt] (A) -- node[pos=0.5, sloped, above, font=\bfseries\Huge, draw=none, fill=none, inner sep=0pt] {1} (D); 
            \draw[red!90!black, very thick, line width = 3pt] (D) -- node[pos=0.5, sloped, below, font=\bfseries\Huge, draw=none, fill=none, inner sep=0pt] {1} (DEmid); 
            \draw[blue!90!black, very thick, line width = 3pt] (BCmid) -- node[pos=0.5, sloped, above, font=\bfseries\Huge, draw=none, fill=none, inner sep=0pt] {1} (C); 
            \draw[blue!90!black, very thick, line width = 3pt] (C) -- node[pos=0.4, sloped, above, font=\bfseries\Huge, draw=none, fill=none, inner sep=0pt] {1} (A); 
            \draw[blue!90!black, very thick, line width = 3pt] (A) -- node[pos=0.5, sloped, above, font=\bfseries\Huge, draw=none, fill=none, inner sep=0pt] {1} (E); 
            \draw[blue!90!black, very thick, line width = 3pt] (E) -- node[pos=0.5, sloped, below, font=\bfseries\Huge, draw=none, fill=none, inner sep=0pt] {1} (DEmid); 
            \draw[green!70!black, very thick, line width = 3pt] (BEmid) -- node[pos=0.3, sloped, below, font=\bfseries\Huge, draw=none, fill=none, inner sep=0pt] {1} (B); 
            \draw[green!70!black, very thick, line width = 3pt] (B) -- node[pos=0.5, sloped, above, font=\bfseries\Huge, draw=none, fill=none, inner sep=0pt] {2} (D); 
            \draw[green!70!black, very thick, line width = 3pt] (D) -- node[pos=0.5, sloped, below, font=\bfseries\Huge, draw=none, fill=none, inner sep=0pt] {1} (CDmid); 
            \draw[yellow!70!black, very thick, line width = 3pt] (CDmid) -- node[pos=0.3, sloped, below, font=\bfseries\Huge, draw=none, fill=none, inner sep=0pt] {1} (C); 
            \draw[yellow!70!black, very thick, line width = 3pt] (C) -- node[pos=0.5, sloped, above, font=\bfseries\Huge, draw=none, fill=none, inner sep=0pt] {2} (E); 
            \draw[yellow!70!black, very thick, line width = 3pt] (E) -- node[pos=0.5, sloped, below, font=\bfseries\Huge, draw=none, fill=none, inner sep=0pt] {1} (BEmid); 
        \end{tikzpicture}
    }
    \hfill 
    \subfloat[]{
        \begin{tikzpicture}[scale=0.35, every node/.style={transform shape, circle, draw=black!60, fill=black!30, minimum size=1cm, font=\bfseries}]
            \node (A) at (0,0) {};
            \node (B) [below left = 5.775cm and 0.825cm of A] {};
            \node (C) [below right = 5.775cm and 0.825cm of A] {};
            \node (D) [below left = 10.45cm and 6.05cm of A] {};
            \node (E) [below right = 10.45cm and 6.05cm of A] {};
            \draw (A) -- (B);
            \draw (A) -- (C);
            \draw (A) -- (D);
            \draw (A) -- (E);
            \draw (B) -- (D);
            \draw (B) -- (C); 
            \draw (B) -- (E); 
            \draw (C) -- (E);
            \draw (C) -- (D); 
            \draw (D) -- (E); 
            \path (B) -- (C) coordinate[pos=0.5] (BCmid);
            \draw[red!90!black, very thick] ($(BCmid)+(0,0.5)$) -- ($(BCmid)+(0,-0.5)$);
            \path (C) -- (E) coordinate[pos=0.5] (CEmid);
            \draw[green!70!black, very thick] ($(CEmid)+(0,0.5)$) -- ($(CEmid)+(0,-0.5)$);
            \path (B) -- (D) coordinate[pos=0.5] (BDmid);
            \draw[yellow!70!black, very thick] ($(BDmid)+(0,0.5)$) -- ($(BDmid)+(0,-0.5)$);
            \path (D) -- (E) coordinate[pos=0.5] (DEmid);
            \draw[blue!90!black, very thick] ($(DEmid)+(0,0.5)$) -- ($(DEmid)+(0,-0.5)$);
            \draw[red!90!black, very thick, line width = 3pt] (BCmid) -- node[pos=0.5, sloped, above, font=\bfseries\Huge, draw=none, fill=none, inner sep=0pt] {1} (B); 
            \draw[red!90!black, very thick, line width = 3pt] (B) -- node[pos=0.4, sloped, above, font=\bfseries\Huge, draw=none, fill=none, inner sep=0pt] {1} (A); 
            \draw[red!90!black, very thick, line width = 3pt] (A) -- node[pos=0.5, sloped, above, font=\bfseries\Huge, draw=none, fill=none, inner sep=0pt] {1} (D); 
            \draw[red!90!black, very thick, line width = 3pt] (D) -- node[pos=0.5, sloped, below, font=\bfseries\Huge, draw=none, fill=none, inner sep=0pt] {1} (DEmid); 
            \draw[blue!90!black, very thick, line width = 3pt] (BCmid) -- node[pos=0.5, sloped, above, font=\bfseries\Huge, draw=none, fill=none, inner sep=0pt] {1} (C); 
            \draw[blue!90!black, very thick, line width = 3pt] (C) -- node[pos=0.4, sloped, above, font=\bfseries\Huge, draw=none, fill=none, inner sep=0pt] {1} (A); 
            \draw[blue!90!black, very thick, line width = 3pt] (A) -- node[pos=0.5, sloped, above, font=\bfseries\Huge, draw=none, fill=none, inner sep=0pt] {1} (E); 
            \draw[blue!90!black, very thick, line width = 3pt] (E) -- node[pos=0.5, sloped, below, font=\bfseries\Huge, draw=none, fill=none, inner sep=0pt] {1} (DEmid); 
            \draw[green!70!black, very thick, line width = 3pt] (CEmid) -- node[pos=0.5, sloped, above, font=\bfseries\Huge, draw=none, fill=none, inner sep=0pt] {1} (C); 
            \draw[green!70!black, very thick, line width = 3pt] (C) -- node[pos=0.5, sloped, below, font=\bfseries\Huge, draw=none, fill=none, inner sep=0pt] {2} (D); 
            \draw[green!70!black, very thick, line width = 3pt] (D) -- node[pos=0.9, sloped, above, font=\bfseries\Huge, draw=none, fill=none, inner sep=0pt] {1} (BDmid); 
            \draw[yellow!70!black, very thick, line width = 3pt] (BDmid) -- node[pos=0.5, sloped, above, font=\bfseries\Huge, draw=none, fill=none, inner sep=0pt] {1} (B); 
            \draw[yellow!70!black, very thick, line width = 3pt] (B) -- node[pos=0.5, sloped, below, font=\bfseries\Huge, draw=none, fill=none, inner sep=0pt] {2} (E); 
            \draw[yellow!70!black, very thick, line width = 3pt] (E) -- node[pos=0.9, sloped, above, font=\bfseries\Huge, draw=none, fill=none, inner sep=0pt] {1} (CEmid);
        \end{tikzpicture}
    }
    \caption{The three distinct optimal geodesic coverings of $K_5$} \label{fig:K5}
    \end{figure}
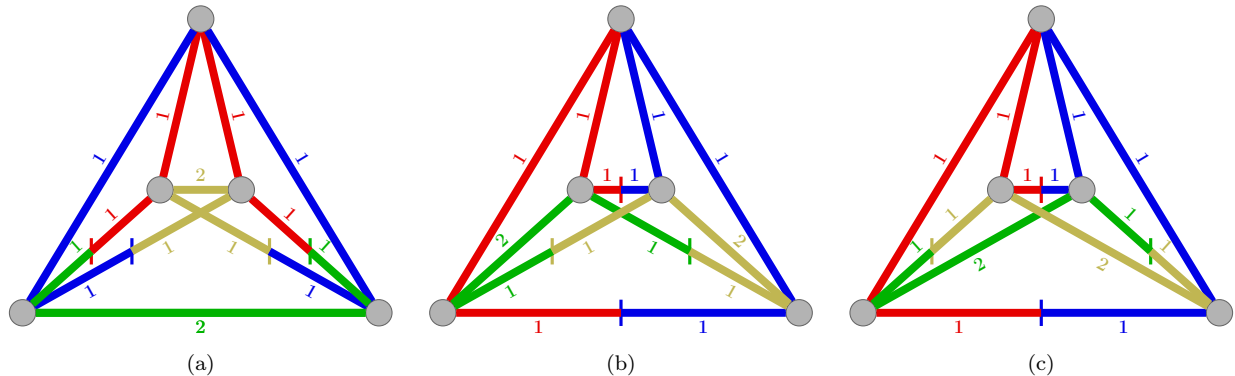

    \subsection{The complete graph $K_4$ on four vertices} \label{sec:K4}

    The cover number of $K_4$ is $4$. Interestingly, this is the same as the larger graph $K_5$. We next show via direct argument that the extended cover number of $K_4$ is also $4$:

    \begin{prop} \label{prop:K4}
        $K_4$ does not appear as a subset of a metric space comprised of $3$ geodesics.
    \end{prop} 
    \begin{proof}
        Assume that there is a space $X$ comprised of $3$ geodesics $X_1, X_2, X_3$ for which $K_4 \subset X$. According to Theorem \ref{prop:3geodesic_catalogue}, either $X$ has a partial order compatible with a consistent orientation on the geodesics or $X$ falls under one of two exceptional cases. In the first case, give $X$  such an order. Define a mapping $\varphi \colon X \to \mathbb{R}$ that is an isometry on each geodesic following the procedure in the proof of Proposition \ref{prop:2geodesics}. Note that necessarily $\varphi$ is monotone on each geodesic $X_i$, and hence the preimage $\varphi^{-1}(a)$ contains at most $3$ values for each $a \in \mathbb{R}$.
        Next, observe that each vertex belongs to at least two geodesics. Thus any two vertices belong to a common geodesic and must be comparable. Label these vertices as $p_1, p_2, p_3, p_4$, where $\varphi(p_1) < \varphi(p_2) < \varphi(p_3) < \varphi(p_4)$. The restriction of $\varphi$ to the edge $(p_i,p_j)$ (for each pair $i \neq j$) is contains the interval $[\varphi(p_i), \varphi(p_j)]$. Hence any $a$ such that $\varphi(p_2) < a < \varphi(p_3)$, the preimage $\varphi^{-1}(a)$ contains at least four point (one from each of the edges $(p_1,p_3)$, $(p_1,p_4)$, $(p_2,p_3)$, $(p_2, p_4)$). This is a contradiction.        
        
        In the second case, we see by inspection that $K_4$ does not appear as a subset of either of the constructions in these exceptional cases. In particular, each vertex of $K_4$ must belong to three edges, each of which leads to another distinct vertex of $K_4$ without crossing through any other vertices of $K_4$. through each of which one can access each of the other vertices of $K_4$ without returning to $a$. By inspection, there are no four points that can be chosen from these spaces with that property.
    \end{proof} 

    There are many covers of $K_4$ by four geodesics. In Figure \ref{fig:K4}, we give one where each graph edge has length one. Further note that all geodesics have equal length. 

    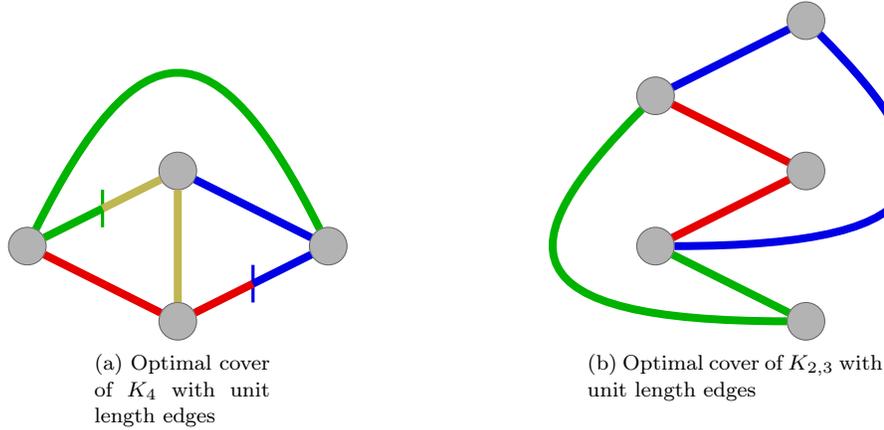
\begin{figure}[t!]
    \centering
    \hfill 
    \subfloat[Optimal cover of $K_4$ with unit length edges]{
        \begin{tikzpicture}[scale=0.5, every node/.style={transform shape, circle, draw=black!60, fill=black!30, minimum size=1cm, font=\bfseries}]
            \node (A) at (0,0) {};
            \node (B) at (4,2) {};
            \node (C) at (4,-2) {};
            \node (D) at (8,0) {};
            \draw[green!70!black, very thick] (2,1.5) -- (2,.5);
            \draw[blue!90!black, very thick] (6,-1.5) -- (6,-.5);
            \draw[green!70!black, very thick, line width = 3pt] (2,1) -- (A) .. controls (3,6) and (5,6) .. (D);
            \draw[red!90!black, very thick, line width = 3pt]  (A) -- (C) -- (6,-1); 
            \draw[blue!90!black, very thick, line width = 3pt] (6,-1) -- (D) -- (B);
            \draw[yellow!70!black, very thick, line width = 3pt] (C) -- (B) -- (2,1); 
        \end{tikzpicture} \label{fig:K4}
    }
    \hfill 
    \subfloat[Optimal cover of $K_{2,3}$ with unit length edges]{
        \begin{tikzpicture}[scale=0.5, every node/.style={transform shape, circle, draw=black!60, fill=black!30, minimum size=1cm, font=\bfseries}]
            \node (A) at (0,0) {};
            \node (B) at (4,2) {};
            \node (C) at (4,-2) {};
            \node (D) at (0,4) {};
            \node (E) at (4,6) {};
            \draw[green!70!black, very thick, line width = 3pt] (A) -- (C) .. controls (-4,-2) and (-4,0) .. (D);
            \draw[red!90!black, very thick, line width = 3pt]  (A) -- (B) -- (D); 
 
            \draw[blue!90!black, very thick, line width = 3pt] (D) -- (E) ..controls (8,2) and (8,0) .. (A); 
        \end{tikzpicture} \label{fig:K23}
    }
    \hfill  \text{ }
    \caption{Geodesic covers for $K_{4}$, $K_{2,3}$ } \label{fig:other_examples}
    \end{figure}

    \subsection{The complete bipartite graph $K_{3,3}$ on two sets of three vertices} \label{sec:K33}

    The graph $K_{3,3}$ has cover number $4$. There are eight distinct retracted geodesic covers up to geodesic rerouting. See Figure \ref{fig:K33}.

    \subsection{The complete bipartite graph $K_{2,3}$ on sets of two and three vertices} \label{sec:K23}

    The graph $K_{2,3}$ has cover number $3$. There are many distinct optimal covers. In Figure \ref{fig:K23}, we give one where each graph edge has length one.

    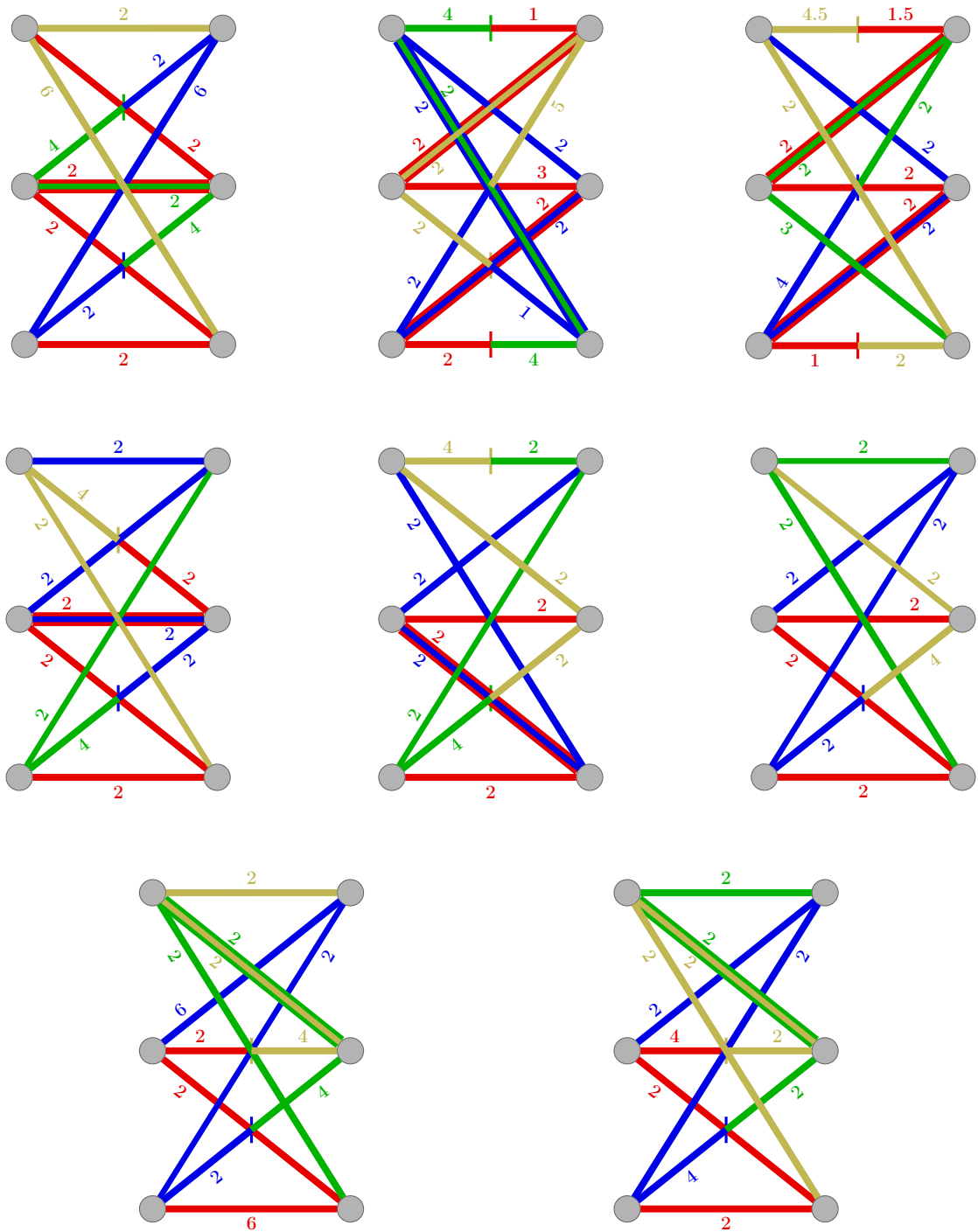
\begin{figure}[t!]
    \centering
    \captionsetup[subfloat]{labelformat=empty, font=small}
    \subfloat[]{
    \begin{minipage}{0.31\textwidth}
        \centering
        \begin{tikzpicture}[scale=0.4, every node/.style={transform shape, circle, draw=black!60, fill=black!30, minimum size=1cm, font=\bfseries}]
            \node (A) at (0,0) {};
            \node (B) [right = 6.5cm of A] {};
            \node (C) [above = 5cm of A] {};
            \node (D) [right = 6.5cm of C] {};
            \node (E) [above = 5cm of C] {};
            \node (F) [right = 6.5cm of E] {};
            \path (C) -- (F) coordinate[pos=0.5] (CFmid);
            \draw[green!70!black, very thick] ($(CFmid)+(0,0.5)$) -- ($(CFmid)+(0,-0.5)$);
            \path (A) -- (D) coordinate[pos=0.5] (ADmid);
            \draw[blue!90!black, very thick] ($(ADmid)+(0,0.5)$) -- ($(ADmid)+(0,-0.5)$);
            \draw[red!90!black, very thick, line width = 3pt] (A) -- node[pos=0.5, sloped, below, font=\bfseries\Huge, draw=none, fill=none, inner sep=0pt] {2} (B); 
            \draw[red!90!black, very thick, line width = 3pt] (B) -- node[pos=0.85, sloped, below, font=\bfseries\Huge, draw=none, fill=none, inner sep=0pt] {2} (C);
            \draw[red!90!black, very thick, line width = 6pt] (C) -- node[pos=0.2, sloped, above, font=\bfseries\Huge, draw=none, fill=none, inner sep=0pt] {2} (D);
            \draw[red!90!black, very thick, line width = 3pt] (D) -- node[pos=0.15, sloped, above, font=\bfseries\Huge, draw=none, fill=none, inner sep=0pt] {2} (E); 
            \draw[blue!90!black, very thick, line width = 3pt] (ADmid) -- node[pos=0.5, sloped, below, font=\bfseries\Huge, draw=none, fill=none, inner sep=0pt] {2} (A);
            \draw[blue!90!black, very thick, line width = 3pt] (A) -- node[pos=0.85, sloped, below, font=\bfseries\Huge, draw=none, fill=none, inner sep=0pt] {6} (F);
            \draw[blue!90!black, very thick, line width = 3pt] (F) -- node[pos=0.5, sloped, above, font=\bfseries\Huge, draw=none, fill=none, inner sep=0pt] {2} (CFmid);
            \draw[green!70!black, very thick, line width = 3pt] (ADmid) -- node[pos=0.69, sloped, below, font=\bfseries\Huge, draw=none, fill=none, inner sep=0pt] {4} (D); 
            \draw[green!70!black, very thick, line width = 2.5pt] (D) -- node[pos=0.2, sloped, below, font=\bfseries\Huge, draw=none, fill=none, inner sep=0pt] {2} (C); 
            \draw[green!70!black, very thick, line width = 3pt] (C) -- node[pos=0.30, sloped, above, font=\bfseries\Huge, draw=none, fill=none, inner sep=0pt] {4} (CFmid); 
            \draw[yellow!70!black, very thick, line width = 3pt] (B) -- node[pos=0.85, sloped, below, font=\bfseries\Huge, draw=none, fill=none, inner sep=0pt] {6} (E);
            \draw[yellow!70!black, very thick, line width = 3pt] (E) -- node[pos=0.5, sloped, above, font=\bfseries\Huge, draw=none, fill=none, inner sep=0pt] {2} (F);
        \end{tikzpicture}
    \end{minipage}
    }
    \hfill
    \subfloat[]{
    \begin{minipage}{0.31\textwidth}
        \centering
        \begin{tikzpicture}[scale=0.4, every node/.style={transform shape, circle, draw=black!60, fill=black!30, minimum size=1cm, font=\bfseries}]
            \node (A) at (0,0) {};
            \node (B) [right = 6.5cm of A] {};
            \node (C) [above = 5cm of A] {};
            \node (D) [right = 6.5cm of C] {};
            \node (E) [above = 5cm of C] {};
            \node (F) [right = 6.5cm of E] {};
            \path (A) -- (B) coordinate[pos=0.5] (ABmid);
            \draw[red!90!black, very thick] ($(ABmid)+(0,0.5)$) -- ($(ABmid)+(0,-0.5)$);
            \path (E) -- (F) coordinate[pos=0.5] (EFmid);
            \draw[green!70!black, very thick] ($(EFmid)+(0,0.5)$) -- ($(EFmid)+(0,-0.5)$);
            \path (B) -- (C) coordinate[pos=0.5] (BCmid);
            \draw[yellow!70!black, very thick] ($(BCmid)+(0,0.5)$) -- ($(BCmid)+(0,-0.5)$);
            \path (A) -- (F) coordinate[pos=0.5] (AFmid);
            \draw[blue!90!black, very thick] ($(AFmid)+(0,0.5)$) -- ($(AFmid)+(0,-0.5)$);
            \draw[red!90!black, very thick, line width = 3pt] (ABmid) -- node[pos=0.5, sloped, below, font=\bfseries\Huge, draw=none, fill=none, inner sep=0pt] {2} (A); 
            \draw[red!90!black, very thick, line width = 6pt] (A) -- node[pos=0.85, sloped, above, font=\bfseries\Huge, draw=none, fill=none, inner sep=0pt] {2} (D);
            \draw[red!90!black, very thick, line width = 3pt] (C) -- node[pos=0.8, sloped, above, font=\bfseries\Huge, draw=none, fill=none, inner sep=0pt] {3} (D);
            \draw[red!90!black, very thick, line width = 6pt] (C) -- node[pos=0.15, sloped, above, font=\bfseries\Huge, draw=none, fill=none, inner sep=0pt] {2} (F); 
            \draw[red!90!black, very thick, line width = 3pt] (F) -- node[pos=0.5, sloped, above, font=\bfseries\Huge, draw=none, fill=none, inner sep=0pt] {1} (EFmid); 
            \draw[blue!90!black, very thick, line width = 3pt] (AFmid) -- node[pos=0.7, sloped, above, font=\bfseries\Huge, draw=none, fill=none, inner sep=0pt] {2} (A);
            \draw[blue!90!black, very thick, line width = 2.5pt] (A) -- node[pos=0.85, sloped, below, font=\bfseries\Huge, draw=none, fill=none, inner sep=0pt] {2} (D);
            \draw[blue!90!black, very thick, line width = 3pt] (D) -- node[pos=0.15, sloped, above, font=\bfseries\Huge, draw=none, fill=none, inner sep=0pt] {2} (E);
            \draw[blue!90!black, very thick, line width = 6pt] (E) -- node[pos=0.2, sloped, below, font=\bfseries\Huge, draw=none, fill=none, inner sep=0pt] {2} (B);
            \draw[blue!90!black, very thick, line width = 3pt] (B) -- node[pos=0.5, sloped, below, font=\bfseries\Huge, draw=none, fill=none, inner sep=0pt] {1} (BCmid);
            \draw[green!70!black, very thick, line width = 3pt] (ABmid) -- node[pos=0.5, sloped, below, font=\bfseries\Huge, draw=none, fill=none, inner sep=0pt] {4} (B); 
            \draw[green!70!black, very thick, line width = 2.5pt] (B) -- node[pos=0.8, sloped, above, font=\bfseries\Huge, draw=none, fill=none, inner sep=0pt] {2} (E); 
            \draw[green!70!black, very thick, line width = 3pt] (E) -- node[pos=0.5, sloped, above, font=\bfseries\Huge, draw=none, fill=none, inner sep=0pt] {4} (EFmid); 
            \draw[yellow!70!black, very thick, line width = 3pt] (AFmid) -- node[pos=0.6, sloped, below, font=\bfseries\Huge, draw=none, fill=none, inner sep=0pt] {5} (F);
            \draw[yellow!70!black, very thick, line width = 2.5pt] (F) -- node[pos=0.85, sloped, below, font=\bfseries\Huge, draw=none, fill=none, inner sep=0pt] {2} (C);
            \draw[yellow!70!black, very thick, line width = 3pt] (C) -- node[pos=0.3, sloped, below, font=\bfseries\Huge, draw=none, fill=none, inner sep=0pt] {2} (BCmid);
        \end{tikzpicture}
    \end{minipage}
    }
    \hfill 
    \subfloat[]{
    \begin{minipage}{0.31\textwidth}
        \centering
        \begin{tikzpicture}[scale=0.4, every node/.style={transform shape, circle, draw=black!60, fill=black!30, minimum size=1cm, font=\bfseries}]
            \node (A) at (0,0) {};
            \node (B) [right = 6.5cm of A] {};
            \node (C) [above = 5cm of A] {};
            \node (D) [right = 6.5cm of C] {};
            \node (E) [above = 5cm of C] {};
            \node (F) [right = 6.5cm of E] {};
            \path (A) -- (B) coordinate[pos=0.5] (ABmid);
            \draw[red!90!black, very thick] ($(ABmid)+(0,0.5)$) -- ($(ABmid)+(0,-0.5)$);
            \path (E) -- (F) coordinate[pos=0.5] (EFmid);
            \draw[yellow!70!black, very thick] ($(EFmid)+(0,0.5)$) -- ($(EFmid)+(0,-0.5)$);
            \path (A) -- (F) coordinate[pos=0.5] (AFmid);
            \draw[blue!90!black, very thick] ($(AFmid)+(0,0.5)$) -- ($(AFmid)+(0,-0.5)$);
            \draw[red!90!black, very thick, line width = 3pt] (ABmid) -- node[pos=0.5, sloped, below, font=\bfseries\Huge, draw=none, fill=none, inner sep=0pt] {1} (A); 
            \draw[red!90!black, very thick, line width = 6pt] (A) -- node[pos=0.85, sloped, above, font=\bfseries\Huge, draw=none, fill=none, inner sep=0pt] {2} (D);
            \draw[red!90!black, very thick, line width = 3pt] (C) -- node[pos=0.8, sloped, above, font=\bfseries\Huge, draw=none, fill=none, inner sep=0pt] {2} (D);
            \draw[red!90!black, very thick, line width = 6pt] (C) -- node[pos=0.15, sloped, above, font=\bfseries\Huge, draw=none, fill=none, inner sep=0pt] {2} (F); 
            \draw[red!90!black, very thick, line width = 3pt] (F) -- node[pos=0.5, sloped, above, font=\bfseries\Huge, draw=none, fill=none, inner sep=0pt] {1.5} (EFmid); 
            \draw[blue!90!black, very thick, line width = 3pt] (AFmid) -- node[pos=0.7, sloped, above, font=\bfseries\Huge, draw=none, fill=none, inner sep=0pt] {4} (A);
            \draw[blue!90!black, very thick, line width = 2.5pt] (A) -- node[pos=0.85, sloped, below, font=\bfseries\Huge, draw=none, fill=none, inner sep=0pt] {2} (D);
            \draw[blue!90!black, very thick, line width = 3pt] (D) -- node[pos=0.15, sloped, above, font=\bfseries\Huge, draw=none, fill=none, inner sep=0pt] {2} (E);
            \draw[green!70!black, very thick, line width = 3pt] (AFmid) -- node[pos=0.6, sloped, below, font=\bfseries\Huge, draw=none, fill=none, inner sep=0pt] {2} (F);
            \draw[green!70!black, very thick, line width = 2.5pt] (F) -- node[pos=0.85, sloped, below, font=\bfseries\Huge, draw=none, fill=none, inner sep=0pt] {2} (C);
            \draw[green!70!black, very thick, line width = 3pt] (C) -- node[pos=0.15, sloped, below, font=\bfseries\Huge, draw=none, fill=none, inner sep=0pt] {3} (B);
            \draw[yellow!70!black, very thick, line width = 3pt] (ABmid) -- node[pos=0.5, sloped, below, font=\bfseries\Huge, draw=none, fill=none, inner sep=0pt] {2} (B);
            \draw[yellow!70!black, very thick, line width = 3pt] (B) -- node[pos=0.805, sloped, below, font=\bfseries\Huge, draw=none, fill=none, inner sep=0pt] {2} (E);
            \draw[yellow!70!black, very thick, line width = 3pt] (E) -- node[pos=0.5, sloped, above, font=\bfseries\Huge, draw=none, fill=none, inner sep=0pt] {4.5} (EFmid);
        \end{tikzpicture}
    \end{minipage}
    }
    
    \subfloat[]{
    \begin{minipage}{0.3\textwidth}
        \centering
        \begin{tikzpicture}[scale=0.4, every node/.style={transform shape, circle, draw=black!60, fill=black!30, minimum size=1cm, font=\bfseries}]
            \node (A) at (0,0) {};
            \node (B) [right = 6.5cm of A] {};
            \node (C) [above = 5cm of A] {};
            \node (D) [right = 6.5cm of C] {};
            \node (E) [above = 5cm of C] {};
            \node (F) [right = 6.5cm of E] {};
            \path (D) -- (E) coordinate[pos=0.5] (DEmid);
            \draw[yellow!70!black, very thick] ($(DEmid)+(0,0.5)$) -- ($(DEmid)+(0,-0.5)$);
            \path (A) -- (D) coordinate[pos=0.5] (ADmid);
            \draw[blue!90!black, very thick] ($(ADmid)+(0,0.5)$) -- ($(ADmid)+(0,-0.5)$);
            \draw[red!90!black, very thick, line width = 3pt] (A) -- node[pos=0.5, sloped, below, font=\bfseries\Huge, draw=none, fill=none, inner sep=0pt] {2} (B); 
            \draw[red!90!black, very thick, line width = 3pt] (B) -- node[pos=0.85, sloped, below, font=\bfseries\Huge, draw=none, fill=none, inner sep=0pt] {2} (C);
            \draw[red!90!black, very thick, line width = 6pt] (C) -- node[pos=0.2, sloped, above, font=\bfseries\Huge, draw=none, fill=none, inner sep=0pt] {2} (D);
            \draw[red!90!black, very thick, line width = 3pt] (D) -- node[pos=0.3, sloped, above, font=\bfseries\Huge, draw=none, fill=none, inner sep=0pt] {2} (DEmid); 
            \draw[blue!90!black, very thick, line width = 3pt] (ADmid) -- node[pos=0.7, sloped, below, font=\bfseries\Huge, draw=none, fill=none, inner sep=0pt] {2} (D);
            \draw[blue!90!black, very thick, line width = 2.5pt] (D) -- node[pos=0.2, sloped, below, font=\bfseries\Huge, draw=none, fill=none, inner sep=0pt] {2} (C);
            \draw[blue!90!black, very thick, line width = 3pt] (C) -- node[pos=0.15, sloped, above, font=\bfseries\Huge, draw=none, fill=none, inner sep=0pt] {2} (F);
            \draw[blue!90!black, very thick, line width = 3pt] (F) -- node[pos=0.5, sloped, above, font=\bfseries\Huge, draw=none, fill=none, inner sep=0pt] {2} (E);
            \draw[green!70!black, very thick, line width = 3pt] (ADmid) -- node[pos=0.5, sloped, below, font=\bfseries\Huge, draw=none, fill=none, inner sep=0pt] {4} (A); 
            \draw[green!70!black, very thick, line width = 2.5pt] (A) -- node[pos=0.15, sloped, above, font=\bfseries\Huge, draw=none, fill=none, inner sep=0pt] {2} (F); 
            \draw[yellow!70!black, very thick, line width = 3pt] (DEmid) -- node[pos=0.5, sloped, above, font=\bfseries\Huge, draw=none, fill=none, inner sep=0pt] {4} (E);
            \draw[yellow!70!black, very thick, line width = 2.5pt] (E) -- node[pos=0.15, sloped, below, font=\bfseries\Huge, draw=none, fill=none, inner sep=0pt] {2} (B);
        \end{tikzpicture}
    \end{minipage}
    }
    \hfill 
    \subfloat[]{
        \begin{minipage}{0.3\textwidth}
        \centering
        \begin{tikzpicture}[scale=0.4, every node/.style={transform shape, circle, draw=black!60, fill=black!30, minimum size=1cm, font=\bfseries}]
            \node (A) at (0,0) {};
            \node (B) [right = 6.5cm of A] {};
            \node (C) [above = 5cm of A] {};
            \node (D) [right = 6.5cm of C] {};
            \node (E) [above = 5cm of C] {};
            \node (F) [right = 6.5cm of E] {};
            \path (A) -- (D) coordinate[pos=0.5] (ADmid);
            \draw[green!70!black, very thick] ($(ADmid)+(0,0.5)$) -- ($(ADmid)+(0,-0.5)$);
            \path (E) -- (F) coordinate[pos=0.5] (EFmid);
            \draw[yellow!70!black, very thick] ($(EFmid)+(0,0.5)$) -- ($(EFmid)+(0,-0.5)$);
            \draw[red!90!black, very thick, line width = 3pt] (A) -- node[pos=0.5, sloped, below, font=\bfseries\Huge, draw=none, fill=none, inner sep=0pt] {2} (B); 
            \draw[red!90!black, very thick, line width = 6pt] (B) -- node[pos=0.85, sloped, above, font=\bfseries\Huge, draw=none, fill=none, inner sep=0pt] {2} (C);
            \draw[red!90!black, very thick, line width = 3pt] (C) -- node[pos=0.8, sloped, above, font=\bfseries\Huge, draw=none, fill=none, inner sep=0pt] {2} (D); 
            \draw[blue!90!black, very thick, line width = 3pt] (E) -- node[pos=0.15, sloped, below, font=\bfseries\Huge, draw=none, fill=none, inner sep=0pt] {2} (B);
            \draw[blue!90!black, very thick, line width = 2.5pt] (B) -- node[pos=0.85, sloped, below, font=\bfseries\Huge, draw=none, fill=none, inner sep=0pt] {2} (C);
            \draw[blue!90!black, very thick, line width = 3pt] (C) -- node[pos=0.15, sloped, above, font=\bfseries\Huge, draw=none, fill=none, inner sep=0pt] {2} (F);
            \draw[green!70!black, very thick, line width = 3pt] (ADmid) -- node[pos=0.5, sloped, below, font=\bfseries\Huge, draw=none, fill=none, inner sep=0pt] {4} (A);
            \draw[green!70!black, very thick, line width = 2.5pt] (A) -- node[pos=0.15, sloped, above, font=\bfseries\Huge, draw=none, fill=none, inner sep=0pt] {2} (F);
            \draw[green!70!black, very thick, line width = 3pt] (F) -- node[pos=0.5, sloped, above, font=\bfseries\Huge, draw=none, fill=none, inner sep=0pt] {2} (EFmid);
            \draw[yellow!70!black, very thick, line width = 3pt] (ADmid) -- node[pos=0.7, sloped, below, font=\bfseries\Huge, draw=none, fill=none, inner sep=0pt] {2} (D);
            \draw[yellow!70!black, very thick, line width = 3pt] (D) -- node[pos=0.15, sloped, above, font=\bfseries\Huge, draw=none, fill=none, inner sep=0pt] {2} (E);
            \draw[yellow!70!black, very thick, line width = 3pt] (E) -- node[pos=0.5, sloped, above, font=\bfseries\Huge, draw=none, fill=none, inner sep=0pt] {4} (EFmid);
        \end{tikzpicture}
    \end{minipage}
    }
    \hfill
    \subfloat[]{
    \begin{minipage}{0.3\textwidth}
        \centering
        \begin{tikzpicture}[scale=0.4, every node/.style={transform shape, circle, draw=black!60, fill=black!30, minimum size=1cm, font=\bfseries}]
            \node (A) at (0,0) {};
            \node (B) [right = 6.5cm of A] {};
            \node (C) [above = 5cm of A] {};
            \node (D) [right = 6.5cm of C] {};
            \node (E) [above = 5cm of C] {};
            \node (F) [right = 6.5cm of E] {};
            \path (A) -- (D) coordinate[pos=0.5] (ADmid);
            \draw[blue!90!black, very thick] ($(ADmid)+(0,0.5)$) -- ($(ADmid)+(0,-0.5)$);
            \draw[red!90!black, very thick, line width = 3pt] (A) -- node[pos=0.5, sloped, below, font=\bfseries\Huge, draw=none, fill=none, inner sep=0pt] {2} (B); 
            \draw[red!90!black, very thick, line width = 3pt] (B) -- node[pos=0.85, sloped, below, font=\bfseries\Huge, draw=none, fill=none, inner sep=0pt] {2} (C);
            \draw[red!90!black, very thick, line width = 3pt] (C) -- node[pos=0.8, sloped, above, font=\bfseries\Huge, draw=none, fill=none, inner sep=0pt] {2} (D);
            \draw[blue!90!black, very thick, line width = 3pt] (ADmid) -- node[pos=0.5, sloped, below, font=\bfseries\Huge, draw=none, fill=none, inner sep=0pt] {2} (A);
            \draw[blue!90!black, very thick, line width = 2.5pt] (A) -- node[pos=0.85, sloped, below, font=\bfseries\Huge, draw=none, fill=none, inner sep=0pt] {2} (F);
            \draw[blue!90!black, very thick, line width = 3pt] (F) -- node[pos=0.85, sloped, above, font=\bfseries\Huge, draw=none, fill=none, inner sep=0pt] {2} (C);
            \draw[green!70!black, very thick, line width = 3pt] (B) -- node[pos=0.85, sloped, below, font=\bfseries\Huge, draw=none, fill=none, inner sep=0pt] {2} (E); 
            \draw[green!70!black, very thick, line width = 2.5pt] (E) -- node[pos=0.5, sloped, above, font=\bfseries\Huge, draw=none, fill=none, inner sep=0pt] {2} (F); 
            \draw[yellow!70!black, very thick, line width = 3pt] (ADmid) -- node[pos=0.7, sloped, below, font=\bfseries\Huge, draw=none, fill=none, inner sep=0pt] {4} (D);
            \draw[yellow!70!black, very thick, line width = 2.5pt] (D) -- node[pos=0.15, sloped, above, font=\bfseries\Huge, draw=none, fill=none, inner sep=0pt] {2} (E);
        \end{tikzpicture}
    \end{minipage}
    }

    \hfill 
    \subfloat[]{
    \begin{minipage}{0.3\textwidth}
        \centering
        \begin{tikzpicture}[scale=0.4, every node/.style={transform shape, circle, draw=black!60, fill=black!30, minimum size=1cm, font=\bfseries}]
            \node (A) at (0,0) {};
            \node (B) [right = 6.5cm of A] {};
            \node (C) [above = 5cm of A] {};
            \node (D) [right = 6.5cm of C] {};
            \node (E) [above = 5cm of C] {};
            \node (F) [right = 6.5cm of E] {};
            \path (A) -- (D) coordinate[pos=0.5] (ADmid);
            \draw[blue!90!black, very thick] ($(ADmid)+(0,0.5)$) -- ($(ADmid)+(0,-0.5)$);
            \path (C) -- (D) coordinate[pos=0.5] (CDmid);
            \draw[yellow!70!black, very thick] ($(CDmid)+(0,0.5)$) -- ($(CDmid)+(0,-0.5)$);
            \draw[red!90!black, very thick, line width = 3pt] (A) -- node[pos=0.5, sloped, below, font=\bfseries\Huge, draw=none, fill=none, inner sep=0pt] {6} (B); 
            \draw[red!90!black, very thick, line width = 3pt] (B) -- node[pos=0.85, sloped, below, font=\bfseries\Huge, draw=none, fill=none, inner sep=0pt] {2} (C);
            \draw[red!90!black, very thick, line width = 3pt] (C) -- node[pos=0.4, sloped, above, font=\bfseries\Huge, draw=none, fill=none, inner sep=0pt] {2} (CDmid); 
            \draw[blue!90!black, very thick, line width = 3pt] (ADmid) -- node[pos=0.5, sloped, below, font=\bfseries\Huge, draw=none, fill=none, inner sep=0pt] {2} (A);
            \draw[blue!90!black, very thick, line width = 2.5pt] (A) -- node[pos=0.85, sloped, below, font=\bfseries\Huge, draw=none, fill=none, inner sep=0pt] {2} (F);
            \draw[blue!90!black, very thick, line width = 3pt] (F) -- node[pos=0.85, sloped, above, font=\bfseries\Huge, draw=none, fill=none, inner sep=0pt] {6} (C);
            \draw[green!70!black, very thick, line width = 3pt] (ADmid) -- node[pos=0.7, sloped, below, font=\bfseries\Huge, draw=none, fill=none, inner sep=0pt] {4} (D);
            \draw[green!70!black, very thick, line width = 6pt] (D) -- node[pos=0.65, sloped, above, font=\bfseries\Huge, draw=none, fill=none, inner sep=0pt] {2} (E);
            \draw[green!70!black, very thick, line width = 3pt] (E) -- node[pos=0.15, sloped, below, font=\bfseries\Huge, draw=none, fill=none, inner sep=0pt] {2} (B);
            \draw[yellow!70!black, very thick, line width = 3pt] (CDmid) -- node[pos=0.6, sloped, above, font=\bfseries\Huge, draw=none, fill=none, inner sep=0pt] {4} (D);
            \draw[yellow!70!black, very thick, line width = 2.5pt] (D) -- node[pos=0.65, sloped, below, font=\bfseries\Huge, draw=none, fill=none, inner sep=0pt] {2} (E);
            \draw[yellow!70!black, very thick, line width = 3pt] (E) -- node[pos=0.5, sloped, above, font=\bfseries\Huge, draw=none, fill=none, inner sep=0pt] {2} (F);
        \end{tikzpicture}
    \end{minipage}
    }
    \hfill
    \subfloat[]{
    \begin{minipage}{0.3\textwidth}
        \centering
        \begin{tikzpicture}[scale=0.4, every node/.style={transform shape, circle, draw=black!60, fill=black!30, minimum size=1cm, font=\bfseries}]
            \node (A) at (0,0) {};
            \node (B) [right = 6.5cm of A] {};
            \node (C) [above = 5cm of A] {};
            \node (D) [right = 6.5cm of C] {};
            \node (E) [above = 5cm of C] {};
            \node (F) [right = 6.5cm of E] {};
            \path (A) -- (D) coordinate[pos=0.5] (ADmid);
            \draw[blue!90!black, very thick] ($(ADmid)+(0,0.5)$) -- ($(ADmid)+(0,-0.5)$);
            \path (C) -- (D) coordinate[pos=0.5] (CDmid);
            \draw[yellow!70!black, very thick] ($(CDmid)+(0,0.5)$) -- ($(CDmid)+(0,-0.5)$);
             \draw[red!90!black, very thick, line width = 3pt] (A) -- node[pos=0.5, sloped, below, font=\bfseries\Huge, draw=none, fill=none, inner sep=0pt] {2} (B); 
            \draw[red!90!black, very thick, line width = 3pt] (B) -- node[pos=0.85, sloped, below, font=\bfseries\Huge, draw=none, fill=none, inner sep=0pt] {2} (C);
            \draw[red!90!black, very thick, line width = 3pt] (C) -- node[pos=0.4, sloped, above, font=\bfseries\Huge, draw=none, fill=none, inner sep=0pt] {4} (CDmid); 
            \draw[blue!90!black, very thick, line width = 3pt] (ADmid) -- node[pos=0.5, sloped, below, font=\bfseries\Huge, draw=none, fill=none, inner sep=0pt] {4} (A);
            \draw[blue!90!black, very thick, line width = 3pt] (A) -- node[pos=0.85, sloped, below, font=\bfseries\Huge, draw=none, fill=none, inner sep=0pt] {2} (F);
            \draw[blue!90!black, very thick, line width = 3pt] (F) -- node[pos=0.85, sloped, above, font=\bfseries\Huge, draw=none, fill=none, inner sep=0pt] {2} (C);
            \draw[green!70!black, very thick, line width = 3pt] (ADmid) -- node[pos=0.7, sloped, below, font=\bfseries\Huge, draw=none, fill=none, inner sep=0pt] {2} (D); 
            \draw[green!70!black, very thick, line width = 6pt] (D) -- node[pos=0.65, sloped, above, font=\bfseries\Huge, draw=none, fill=none, inner sep=0pt] {2} (E); 
            \draw[green!70!black, very thick, line width = 3pt] (E) -- node[pos=0.5, sloped, above, font=\bfseries\Huge, draw=none, fill=none, inner sep=0pt] {2} (F); 
            \draw[yellow!70!black, very thick, line width = 3pt] (CDmid) -- node[pos=0.6, sloped, above, font=\bfseries\Huge, draw=none, fill=none, inner sep=0pt] {2} (D);
            \draw[yellow!70!black, very thick, line width = 2.5pt] (D) -- node[pos=0.65, sloped, below, font=\bfseries\Huge, draw=none, fill=none, inner sep=0pt] {2} (E);
            \draw[yellow!70!black, very thick, line width = 3pt] (E) -- node[pos=0.15, sloped, below, font=\bfseries\Huge, draw=none, fill=none, inner sep=0pt] {2} (B);
        \end{tikzpicture} 
    \end{minipage}
    }
    \hfill \,
    \caption{The eight distinct optimal geodesic covers of $K_{3,3}$} \label{fig:K33}
    \end{figure}

\section*{Appendix A: Algorithm and implementation}


In this appendix, we include details of an algorithm to determine the cover number of a graph. It proceeds in two parts. First, we compile a list of candidate covers of size at most some predetermined \(M\). Recall from Theorem \ref{thm:2subdivision} that we can take all endpoints to be graph vertices of the $2$-subdivision of the original graph. The core subroutine of this step recursively builds subcovers, filtering out any nonretracted subcovers at each step. 

Second, for each candidate cover we test whether it can be realizes as a geodesic cover with some weighting of the segments. This is done by expressing the configuration as a linear program with the weights as the variables and the condition that each path in the cover is a shortest path as a system of linear constraints. If we find a geodesic cover, then the geodesic cover number is shown to be at most \(M\); otherwise, it is more than \(M\). This process is repeated for increasing values of \(M\) until a geodesic cover is found.

\begin{algorithm}[H]
    \caption{FindSubCovers$(s)$}\label{alg:findsubcovers}
    \begin{algorithmic}[1]
    \Require current partial cover $s$; pool of paths $P$; max size $M$; global list $C$
    \If{\Call{CoversAllSegments}{$s$}}
        \State append $s$ to $C$ \Comment complete cover found
        \State \Return
    \ElsIf{$|s| = M$}
        \State \Return
    \EndIf

    \ForAll{$p \in P$}
        \If{$\exists g\in s:p[0] \in g \lor\ p[-1] \in g$} \Comment ensure retraction property satisfied
            \State \textbf{continue}
        \ElsIf{$\exists g\in s:g[0]\in p\lor g[-1]\in p$}
            \State \textbf{continue}
        \EndIf
        \State $s' \gets s \cup \{p\}$
        \State \Call{FindSubCovers}{$s'$}
    \EndFor
    \end{algorithmic}
\end{algorithm}

Once we have the potential covers found by the above algorithm, we further cut down the number we need to consider for the linear program by discarding extra covers that are equivalent up to an automorphism of the graph. We only keep covers that are minimal (with respect to some arbitrary total order) among all equivalent covers up to automorphism. This is done by checking if any permutation in the automorphism group (which is precomputed) results in a smaller cover, and if so, discarding the original.

\begin{algorithm}[H]
\caption{IsMinimalInSymmetries$(c)$}\label{alg:isminimalinsymmetries}
\begin{algorithmic}[1]
\Require cover $c$, precomputed symmetry group $\mathcal{S}=\mathrm{Symmetries}(G)$, total order $<$ of covers
\ForAll{$\sigma \in \mathcal{S}$}
    \If{$\Call{ApplyPermutation}{\sigma, c} < c$}
        \State \Return \textbf{false}
    \EndIf
\EndFor
\State \Return \textbf{true}
\end{algorithmic}
\end{algorithm}

Finally, we optionally discard coverings that are equivalent up to geodesic rerouting. For this, we iteratively update the collection of reroutings by seeing if a local rerouting move results in a new cover (in which case it is added to \(Q\)), and continuing this process until all covers formed by reroutings have been checked. We compare new covers at each step to the starting cover, and if the starting cover is not smaller (with respect to the arbitrary total order), the cover is disqualified.

\begin{algorithm}[H]
\caption{IsMinimalInReroutings$(c)$}\label{alg:isminimalinreroutings}
\begin{algorithmic}[1]
\Require cover $c$; total order $<$ of covers; iterator \Call{Reroutings}{$\cdot$} giving the (finite) neighbors under one rerouting move
\State $\textit{Visited} \gets \{c\}$
\State $\textit{Q} \gets \Call{Queue}{}$
\State \Call{Enqueue}{$\textit{Q}, c$}
\While{\textbf{not} \Call{Empty}{$\textit{Q}$}}
    \State $x \gets \Call{Dequeue}{\textit{Q}}$
    \ForAll{$y \in \Call{Reroutings}{x}$}
        \If{$y \notin \textit{Visited}$}
            \If{$y < c$}
                \State \Return \textbf{false}
            \EndIf
            \State add $y$ to $\textit{Visited}$
            \State \Call{Enqueue}{$\textit{Q}, y$}
        \EndIf
    \EndFor
\EndWhile
\State \Return \textbf{true}
\end{algorithmic}
\end{algorithm}

The overall algorithm to find candidate covers produces an initial list of candidate covers with the first subroutine, then filters that list with the second and third.

For determining whether a given cover can be made into a geodesic cover with some weighting of the segments, we first write those segments as variables constrained to be at least one. This is to ensure strictly positive distances, and we note that it does not change feasibility of the resulting linear program. We set the objective to minimize the sum over these variables. Note that other objectives would also work; we are only interested in the feasibility of a solution rather than the optimality. Finally, a linear program solver is applied to verify if the cover can be made into a geodesic cover.

\begin{algorithm}[H]
\caption{TestGeodesicFeasibility$(c)$}
\begin{algorithmic}[1]
\Require candidate cover $c$ (tuple of paths); global segment set $S$; global path list $P$
\State create LP $\mathcal{P}$ with decision variables $w[s]$ for each $s \in S$; each has lower bound $1$
\State set objective of $\mathcal{P}$ to minimize $\sum_{s \in S} w[s]$
\ForAll{paths $g$ in $c$}
  \ForAll{paths $p$ in $P$ with $p[0] = g[0]$ and $p[-1] = g[-1]$}
    \State add constraint:
    \State \hspace{1em} sum\{$w[s]:s\in$\Call{SegmentsFromPath}{$g$}\}
           $\le$
           sum\{$w[s]:s\in$\Call{SegmentsFromPath}{$p$}\}
  \EndFor
\EndFor
\State solve $\mathcal{P}$
\end{algorithmic}
\end{algorithm}

\section*{Appendix B: Configurations of three paths not compatible with geodesics}

Here, we give the remaining details of the proof of Proposition \ref{prop:3geodesic_catalogue}. Namely, we show that any configuration of paths other than those stated in Proposition \ref{prop:3geodesic_catalogue} do not admit a length metric in which $X_1, X_2, X_3$ are geodesics. 
We recall that we may assume that these geodesics satisfy $|X_1 \cap X_2| \geq 2$ and $|X_1 \cap X_3| \geq 2$, with each pair $X_1, X_2$ and $X_1, X_3$ being compatibly oriented although $X$ itself does not have a partial order compatible with these orientations. There are two ways this can happen: (1) there exist points $a \in X_1 \cap X_3$, $b \in X_1 \cap X_2$ and $c \in X_2 \cap X_3$ such that $a \prec_1 b$, $b \prec_2 c$ and $c \prec_3 a$, and (2) there exist points $a,b \in X_2 \cap X_3$ such that $a \prec_3 b$ and $b \prec_2 a$. These points force $X$ to take one of the configurations below that are shown to be admissible. Any additional points in the intersections $X_1 \cap X_2$, $X_1 \cap X_3$, $X_2 \cap X_3$ must also be compatible with this same configuration. 
We consider each case separately. 


\subsection*{Group 1} Assume there exist points $a \in X_1 \cap X_3$, $b \in X_1 \cap X_2$ and $c \in X_2 \cap X_3$ such that $a \prec_1 b$, $b \prec_2 c$ and $c \prec_3 a$. Observe first that $a \notin X_2$, $b \notin X_3$ and $c \notin X_1$.

    Let $e$ denote a point in $X_1 \cap X_3$ distinct from $a$, which exists by assumption. Let $f$ denote a point in $X_1 \cap X_2$ distinct from $b$, which also exists by assumption. Observe that the points $a,b,e,f$ belong to $X_1$, while $b,c,f$ belong to $X_2$ and $a,c,e$ belong to $X_3$. We note that it's possible for $e=f$. Any arrangement of these points along the geodesics $X_1, X_2, X_3$ gives a valid configuration provided it satisfies the constraints $a \prec_1 b$, $b \prec_2 c$, $c \prec_3 a$, and $b \prec_1 f \Longleftrightarrow b \prec_2 f$ and $a \prec_1 e \Longleftrightarrow a \prec_3 e$. There are $21$ such configurations:

    \begin{multicols}{2}
    \begin{enumerate}
        \item $a \prec_1 b \prec_1 e \preceq_1 f$, \, $b \prec_2 c \prec_2 f$, \, $c \prec_3 a \prec_3 e$
        \item $a \prec_1 b \prec_1 e \preceq_1 f$, \, $b \prec_2 f \prec_2 c$, \, $c \prec_3 a \prec_3 e$
        \item $a \prec_1 b \prec_1 f \preceq_1 e$, \, $b \prec_2 c \prec_2 f$, \, $c \prec_3 a \prec_3 e$
        \item $a \prec_1 b \prec_1 f \preceq_1 e$, \, $b \prec_2 f \prec_2 c$, \, $c \prec_3 a \prec_3 e$
        \item $a \prec_1 e \prec_1 b \prec_1 f$, \, $b \prec_2 c \prec_2 f$, \, $c \prec_3 a \prec_3 e$
        \item $a \prec_1 e \prec_1 b \prec_1 f$, \, $b \prec_2 f \prec_2 c$, \, $c \prec_3 a \prec_3 e$
        \item $a \prec_1 e \preceq_1 f \prec_1 b$, \, $f \prec_2 b \prec_2 c$, \, $c \prec_3 a \prec_3 e$
        \item $a \prec_1 f \prec_1 b \prec_1 e$, \, $f \prec_2 b \prec_2 c$, \, $c \prec_3 a \prec_3 e$
        \item $a \prec_1 f \preceq_1 e \prec_1 b$, \, $f \prec_2 b \prec_2 c$, \, $c \prec_3 a \prec_3 e$
        \item $e \prec_1 a \prec_1 b \prec_1 f$, \, $b \prec_2 c \prec_2 f$, \, $c \prec_3 e \prec_3 a$
        \item $e \prec_1 a \prec_1 b \prec_1 f$, \, $b \prec_2 c \prec_2 f$, \, $e \prec_3 c \prec_3 a$
        \item $e \prec_1 a \prec_1 b \prec_1 f$, \, $b \prec_2 f \prec_2 c$, \, $c \prec_3 e \prec_3 a$
        \item $e \prec_1 a \prec_1 b \prec_1 f$, \, $b \prec_2 f \prec_2 c$, \, $e \prec_3 c \prec_3 a$
        \item $e \prec_1 a \prec_1 f \prec_1 b$, \, $f \prec_2 b \prec_2 c$, \, $c \prec_3 e \prec_3 a$
        \item $e \prec_1 a \prec_1 f \prec_1 b$, \, $f \prec_2 b \prec_2 c$, \, $e \prec_3 c \prec_3 a$
        \item $e \preceq_1 f \prec_1 a \prec_1 b$, \, $f \prec_2 b \prec_2 c$, \, $c \prec_3 e \prec_3 a$
        \item $e \preceq_1 f \prec_1 a \prec_1 b$, \, $f \prec_2 b \prec_2 c$, \, $e \preceq_3 c \prec_3 a$
        \item $f \prec_1 a \prec_1 b \prec_1 e$, \, $f \prec_2 b \prec_2 c$, \, $c \prec_3 a \prec_3 e$
        \item $f \prec_1 a \prec_1 e \prec_1 b$, \, $f \prec_2 b \prec_2 c$, \, $c \prec_3 a \prec_3 a$
        \item $f \preceq_1 e \prec_1 a \prec_1 b$, \, $f \prec_2 b \prec_2 c$, \, $c \prec_3 e \prec_3 a$
        \item $f \preceq_1 e \prec_1 a \prec_1 b$, \, $f \prec_2 b \prec_2 c$, \, $e \prec_3 c \prec_3 a$
    \end{enumerate}
    \end{multicols}
    Each of these configurations can be analyzed by hand or by applying Algorithm 4 in Appendix A. All lead to contradictions except (6), (7), (9), (12) and (14). Moreover, (9) is only possible if $e=f$. All these cases correspond to the configuration in (2a) of Proposition \ref{prop:3geodesic_catalogue} (if $X_2$ and $X_3$ have consistent orientations) or (2b) of the same proposition (if $X_2$ and $X_3$ have inconsistent orientations). Note that it is possible that $X_2$ and $X_3$ intersect in a single point, in which case both (2a) and (2b) reduce to the same configuration. Also note that the equality $e=f$ is permitted in the case of (2b), while in the case of (2a) it leads to a contradiction.

    \subsection*{Group 2} Assume there exist points $a,b \in X_2 \cap X_3$ such that $a \prec_2 b$ and $b \prec_3 a$. Let $c,e$ be distinct points in $X_1 \cap X_2$ and let $f,g$ be distinct points in $X_1 \cap X_3$. By relabeling $c,e$ and $f,g$ if needed, we can assume that $c \prec_1 e$ (and hence $c \prec_2 e$) and that $f \prec_1 g$ (and hence $f \prec_3 g$). We can also assume by symmetry that $f$ is the smallest point in $X_3$ out of $b,a,f,g$ (otherwise, reverse the orientation on $X_2$, put $X_3$ into the role of $X_1$, and relabel vertices as needed to match it with table). Based on these constraints, there are 108 configurations to consider, namely all combinations of the following cases (with the requirement that for any two consecutive inequalities at least one of them must be proper):

    \begin{multicols}{3}
    
    \begin{enumerate}
        \item[(1a)] $c \prec_1 e \preceq_1 f \prec_1 g$
        \item[(1b)] $c \preceq_1 f \preceq_1 e \preceq_1 g$
        \item[(1c)] $c \preceq_1 f \prec_1 g \preceq_1 e$
        \item[(1d)] $f \prec_1 g \preceq_1 c \prec_1 e$
        \item[(1e)] $f \preceq_1 c \preceq_1 g \preceq_1 e$
        \item[(1f)] $f \preceq_1 c \prec_1 e \preceq_1 g$
    \end{enumerate}
    \vfill \columnbreak
    
    \begin{enumerate}
        \item[(2a)] $a \prec_2 b \preceq_2 c \prec_2 e$
        \item[(2b)] $a \preceq_2 c \preceq_2 b \preceq_2 e$
        \item[(2c)] $a \preceq_2 c \prec_2 e \preceq_2 b$
        \item[(2d)] $c \prec_2 e \preceq_2 a \prec_2 b$
        \item[(2e)] $c \preceq_2 a \preceq_2 e \preceq_2 b$
        \item[(2f)] $c \preceq_2 a \prec_2 b \preceq_2 e$
    \end{enumerate}
    \vfill \columnbreak

    \begin{enumerate}
        \item[(3a)] $f \prec_3 g \preceq_3 b \prec_3 a$
        \item[(3b)] $f \preceq_3 b \preceq_3 g \preceq_3 a$
        \item[(3c)] $f \preceq_3 b \prec_3 a \preceq_3 g$
    \end{enumerate}
    \end{multicols}

    In some cases, not all $6$ points are required to reach a contradiction, so the number of cases can be reduced somewhat if desired. In the interest of simplicity, however, we do not pursue this here. By a computer analysis, it can be shown that each of these cases leads to a contradiction except the following seven main cases (where $a,b,c,e,f,g$ are distinct points):

    \begin{tabular}{llll}
        (1) & (1a) $c \prec_1 e \preceq_1 f \prec_1 g$, & (2a) $a \prec_2 b \prec_2 c \prec_2 e$, & (3a) $f \prec_3 g \preceq_3 b \prec_3 a$ \\
        (2) & (1a) $c \prec_1 e \prec_1 f \prec_1 g$, & (2d) $c \prec_2 e \prec_2 a \prec_2 b$, & (3a) $f \prec_3 g \preceq_3 b \prec_3 a$ \\
        (3) & (1c) $c \prec_1 f \prec_1 g \preceq_1 e$, & (2f) $c \prec_2 a \prec_2 b \preceq_2 e$, & (3a) $f \prec_3 g \preceq_3 b \prec_3 a$ \\
        (4) & (1f) $f \preceq_1 c \prec_1 e \prec_1 g$, & (2a) $a \prec_2 b \preceq_2 c \prec_2 e$, & (3c) $f \preceq_3 b \prec_3 a \prec_3 g$ \\
        (5) & (1f) $f \preceq_1 c \prec_1 e \preceq_1 g$, & (2d) $c \prec_2 e \preceq_2 a \prec_2 b$, & (3c) $f \prec_3 b \prec_3 a \preceq_3 g$ \\
        (6) & (1d) $f \prec_1 g \preceq_1 c \prec_1 e$, & (2a) $a \prec_2 b \preceq_2 c \prec_2 e$, & (3a) $f \prec_3 g \preceq_3 b \prec_3 a$ \\
        (7) & (1d) $f \prec_1 g \prec_1 c \prec_1 e$, & (2d) $c \prec_2 e \preceq_2 a \prec_2 b$, & (3a) $f \prec_3 g \prec_3 b \prec_3 a$
    \end{tabular}

    Note that in each case there are some restrictions on when an inequality must be proper. We also note that many other combinations from the above table also yield admissible configurations, but only by identifying certain pairs of points. However, each such ``degenerate'' configuration is a subcase of one of the above seven. 

    We see by inspection that each of the seven configurations is compatible with that in (2b) of Proposition \ref{prop:3geodesic_catalogue}. This completes the proof.

\bibliographystyle{abbrv}  
\bibliography{biblio}

\end{document}